\RequirePackage{fix-cm}
\documentclass[smallcondensed]{svjour3}      
\smartqed  
\usepackage{graphicx}
\usepackage{amssymb}
\usepackage{amsmath}
\usepackage{color}
\usepackage{wrapfig}
\usepackage{float}
\usepackage{epic}   
\usepackage{times}

\allowdisplaybreaks[4] 

\journalname{Acta Applicandae Mathematicae}

\newcommand{\bbF}{{\mathbb F}}
\newcommand{\bbG}{{\mathbb G}}
\newcommand{\bbH}{{\mathbb H}}
\newcommand{\bbI}{{\mathbb I}}

\newcommand{\N}{{\mathbb N}}
\newcommand{\bbO}{{\mathbb O}}

\newcommand{\R}{{\mathbb R}}

\newcommand{\Z}{{\mathbb Z}}

\newcommand{\bfe}{\mathbf{e}}
\newcommand{\bff}{\mathbf{f}}
\newcommand{\bfg}{\mathbf{g}}
\newcommand{\bfh}{\mathbf{h}}

\newcommand{\bfn}{\mathbf{n}}

\newcommand{\bfu}{\mathbf{u}}
\newcommand{\bfv}{\mathbf{v}}
\newcommand{\bfw}{\mathbf{w}}
\newcommand{\bfx}{\mathbf{x}}

\newcommand{\bfF}{\mathbf{F}}
\newcommand{\bfG}{\mathbf{G}}

\newcommand{\bfL}{\mathbf{L}}

\newcommand{\bfV}{\mathbf{V}}
\newcommand{\bfW}{\mathbf{W}}

\newcommand{\cA}{{\cal A}}

\newcommand{\cC}{{\cal C}}

\newcommand{\cF}{{\cal F}}

\newcommand{\cO}{{\cal O}}

\newcommand{\rmd}{\mathrm{d}}
\newcommand{\rme}{\mathrm{e}}

\newcommand{\br}{\hspace{0.7pt}}
\renewcommand{\div}{\mathrm{div}\,}

\renewcommand{\atop}[2]{\genfrac{}{}{0pt}{}{#1}{#2}}
\newcommand{\bfzero}{\mathbf{0}}
\newcommand{\dist}{\mathrm{dist}}

\newcommand{\supp}{\mathrm{supp}\,}

\newcommand{\Gammai}{\Gamma_{\hspace{-1.1pt} \rm in}}
\newcommand{\Gammao}{\Gamma_{\hspace{-1.1pt} \rm out}}
\newcommand{\Gammaw}{\Gamma_{\hspace{-1.1pt} p}}
\newcommand{\Gammam}{\Gamma_{\hspace{-1.1pt} 0}}
\newcommand{\Gammap}{\Gamma_{\hspace{-1.1pt} 1}}
\newcommand{\gammai}{\gamma_{\rm in}}
\newcommand{\gammao}{\gamma_{\rm out}}

\newcommand{\bfgs}{\bfg_{\displaystyle *}}

\newcommand{\cCs}{\boldsymbol{\cC}^{\infty}_{\sigma}(\overline{\Omega})}
\newcommand{\cCns}{\boldsymbol{\cC}^{\infty}_{0,\sigma}(\Omega)}
\newcommand{\cCn}{\boldsymbol{\cC}^{\infty}_0(\Omega)}
\newcommand{\Ls}{\bfL_{\sigma}^2(\Omega)}
\newcommand{\Vs}{\bfV_{\sigma}^{1,2}(\Omega)}
\newcommand{\Vsd}{\bfV_{\sigma}^{-1,2}(\Omega)}
\newcommand{\vsd}{\bfV_{\sigma}^{-1,2}}

\newcommand{\bfvt}{\widetilde{\bfv}}
\newcommand{\vt}{\widetilde{v}}
\newcommand{\pt}{\widetilde{p}}
\newcommand{\bfft}{\widetilde{\bff}}
\newcommand{\bfht}{\widetilde{\bfh}}
\newcommand{\hht}{\widetilde{h}}
\newcommand{\bbFt}{\widetilde{\bbF}}
\newcommand{\bfFt}{\widetilde{\bfF}}
\newcommand{\Omegat}{\widetilde{\Omega}}

\newcommand{\lclosed}{[\hbox to 1pt{}}
\newcommand{\rclosed}{\hbox to 1pt{}]}

\newcommand{\llangle}{\langle\hspace{-1.9pt}\langle}
\newcommand{\rrangle}{\rangle\hspace{-1.9pt}\rangle}
\newcommand{\blangle}{\bigl\langle}
\newcommand{\brangle}{\bigr\rangle}

\newcounter{constants}
\newcommand{\cn}[2]{ \addtocounter{constants}{1}
\newcounter{c#1#2}
\setcounter{c#1#2}{\value{constants}} c_{\arabic{c#1#2}} }
\newcommand{\cc}[2]{c_{\arabic{c#1#2}}}

\definecolor{lightgrey}{rgb}{0.82,0.82,0.82}

\begin{document}

\title{The maximum regularity property of the steady Stokes
problem associated with a flow through a profile cascade}

\titlerunning{Steady Stokes problem associated with a flow through a
profile cascade}

\author{Tom\'a\v{s} Neustupa}

\authorrunning{Tom\'a\v{s} Neustupa}

\institute{Czech Technical University, Faculty of Mechanical
Engineering, Department of Technical Mathematics, Karlovo
n\'am.~13, 121 35 Praha 2, Czech Republic \\
\email{tomas.neustupa@fs.cvut.cz}}

\date{Received: date / Accepted: date}

\maketitle

\begin{abstract}
We deal with a steady Stokes-type problem, associated with a
flow of a Newtonian incompressible fluid through a spatially
periodic profile cascade. The used mathematical model is based
on the reduction to one spatial period, represented by a
bounded 2D domain $\Omega$. The corresponding Stokes--type
problem is formulated by means of the Stokes equation, equation
of continuity and three types of boundary conditions: the
conditions of periodicity on the curves $\Gammam$ and
$\Gammap$, the Dirichlet boundary conditions on $\Gammai$ and
$\Gammaw$ and an artificial ``do nothing''--type boundary
condition on $\Gammao$. (See Fig.~1.) We explain on the level
of weak solutions the sense in which the last condition is
satisfied. We show that, although domain $\Omega$ is not smooth
and different types of boundary conditions meet in the corners
of $\Omega$, the considered problem has a strong solution with
the so called maximum regularity property.

\end{abstract}

\section{Introduction} \label{S1}

\paragraph{The profile cascade and reduction to one spatial period.}
The flow through a 3D turbine wheel is often being modelled by
a flow through a 2D profile cascade, which consists of an
infinite number of profiles that periodically repeat with the
period $\tau\br\bfe_2$ in the $x_2$--direction, see Fig.~1.
Here, we use the planar Cartesian coordinate system $x_1$,
$x_2$. Unit vectors in the directions of the $x_1$ and $x_2$
axes are denoted by $\bfe_1$ and $\bfe_2$, respectively. The
profile cascade consists of an infinite family of profiles
$\{P_k\}_{k\in\Z}$ such that $P_k$ are closed bounded sets in
the stripe $\R^2_{(0,d)}:=\{\bfx\equiv(x_1,x_2)\in\R^2;\
0<x_1<d\}$, with Lipschitzian boundaries, such that $\,
P_k=P_0+k\br\tau\br\bfe_2$ and $P_k\cap P_{k+1}=\emptyset$ for
$k\in\Z$. As the set $\cO:= \{(\bfx=(x_1,x_2)\in\R^2;\
x_1\in(0,d)\} \smallsetminus \cup_{k=-\infty}^{\infty}P_k$,
through which the fluid flows, is spatially periodic, it is
natural to assume that, provided that the acting body force and
the given boundary data are also spatially periodic with the
same period $\tau\br\bfe_2$, the fluid flow is spatially
periodic, too. This enables us to reduce the mathematical model
of the flow through

\begin{wrapfigure}[17]{r}{64mm}
  \setlength{\unitlength}{0.4mm}
  \hspace{-10pt}
  \begin{picture}(132,163)
  \put(5,60){\vector(1,0){157}} \put(20,31){\vector(0,1){124}}
  \put(149,64){\small $x_1$} \put(23,147){\small $x_2$}
  \put(28.5,106.4){\vector(0,1){3.2}}
  \dashline[+30]{2.2}(28.5,104.9)(28.5,75.2)
  \put(28.5,73.6){\vector(0,-1){3.7}} \put(31.5,97){\small $\tau$}
  \thicklines 
  \color{lightgrey}
\put(40.2,90){\line(1,0){35}} \put(40.2,90.4){\line(1,0){34}}
\put(40.2,90.8){\line(1,0){33}}
\put(40.4,91.2){\line(1,0){31.6}}
\put(40.4,91.6){\line(1,0){30.4}}
\put(40.5,92){\line(1,0){29.3}}
\put(40.5,92.4){\line(1,0){28.1}}\put(40.8,92.8){\line(1,0){26.5}}
\put(41.0,93.2){\line(1,0){25.0}}\put(41.5,93.6){\line(1,0){23.3}}
\put(41.8,94.0){\line(1,0){21.3}}\put(42.1,94.4){\line(1,0){19.3}}
\put(42.5,94.8){\line(1,0){17.3}}\put(43.2,95.2){\line(1,0){14.8}}
\put(44.3,95.6){\line(1,0){11.8}}\put(46.3,96.0){\line(1,0){6.8}}
\put(40.2,89.6){\line(1,0){36.1}}\put(40.4,89.2){\line(1,0){36.7}}
\put(40.6,88.8){\line(1,0){37.4}}\put(40.8,88.4){\line(1,0){38.2}}
\put(41.1,88.0){\line(1,0){38.9}}\put(41.5,87.6){\line(1,0){39.4}}
\put(42.1,87.2){\line(1,0){39.5}}\put(42.7,86.8){\line(1,0){39.8}}
\put(43.7,86.4){\line(1,0){39.7}}\put(44.5,86){\line(1,0){39.6}}
\put(45.7,85.6){\line(1,0){39.3}}\put(46.8,85.2){\line(1,0){39.1}}
\put(48.5,84.8){\line(1,0){38.2}}\put(50.4,84.4){\line(1,0){37.0}}
\put(52.3,84.0){\line(1,0){35.8}}\put(55.0,83.6){\line(1,0){33.8}}
\put(58.3,83.2){\line(1,0){31.2}}\put(61.7,82.8){\line(1,0){28.6}}
\put(65.3,82.4){\line(1,0){25.6}}\put(68.0,82.0){\line(1,0){23.5}}
\put(70.2,81.6){\line(1,0){21.9}}\put(72.2,81.2){\line(1,0){20.9}}
\put(74.2,80.8){\line(1,0){19.4}}\put(75.9,80.4){\line(1,0){18.3}}
\put(77.5,80.0){\line(1,0){17.3}}\put(79.0,79.6){\line(1,0){16.3}}

\put(80.8,79.2){\line(1,0){15.3}}\put(82.3,78.8){\line(1,0){14.4}}
\put(84.0,78.4){\line(1,0){13.3}}\put(85.5,78.0){\line(1,0){12.6}}
\put(87,77.6){\line(1,0){11.7}}\put(88.5,77.2){\line(1,0){10.8}}
\put(90.0,76.8){\line(1,0){9.9}}\put(91,76.4){\line(1,0){9.3}}
\put(92.2,76.0){\line(1,0){8.8}}\put(93.4,75.6){\line(1,0){8.2}}
\put(94.4,75.2){\line(1,0){7.4}}\put(95.8,74.8){\line(1,0){6.7}}
\put(97.1,74.4){\line(1,0){5.6}}\put(98.3,74.0){\line(1,0){4.6}}
\put(99.5,73.6){\line(1,0){3.5}}\put(100.3,73.2){\line(1,0){3.0}}
\color{black} 
\qbezier(40,90)(40,100)(63,94) \qbezier(40,90)(40,85)(60,83)
\qbezier(63,94)(86,87)(98,78)
\qbezier(60,83)(80,80.5)(97.2,74.3)
\qbezier(98,78)(109,70)(97.2,74.3) \put(47,88){\small $P_0$}
%
\color{lightgrey} \put(40.2,50){\line(1,0){35}}
\put(40.2,50.4){\line(1,0){34}} \put(40.2,50.8){\line(1,0){33}}
\put(40.4,51.2){\line(1,0){31.6}}
\put(40.4,51.6){\line(1,0){30.4}}
\put(40.5,52){\line(1,0){29.3}}
\put(40.5,52.4){\line(1,0){28.1}}\put(40.8,52.8){\line(1,0){26.5}}
\put(41.0,53.2){\line(1,0){25.0}}\put(41.5,53.6){\line(1,0){23.3}}
\put(41.8,54.0){\line(1,0){21.3}}\put(42.1,54.4){\line(1,0){19.3}}
\put(42.5,54.8){\line(1,0){17.3}}\put(43.2,55.2){\line(1,0){14.8}}
\put(44.3,55.6){\line(1,0){11.8}}\put(46.3,56.0){\line(1,0){6.8}}
\put(40.2,49.6){\line(1,0){36.1}}\put(40.4,49.2){\line(1,0){36.7}}
\put(40.6,48.8){\line(1,0){37.4}}\put(40.8,48.4){\line(1,0){38.2}}
\put(41.1,48.0){\line(1,0){38.9}}\put(41.5,47.6){\line(1,0){39.4}}
\put(42.1,47.2){\line(1,0){39.5}}\put(42.7,46.8){\line(1,0){39.8}}
\put(43.7,46.4){\line(1,0){39.7}}\put(44.5,46){\line(1,0){39.6}}
\put(45.7,45.6){\line(1,0){39.3}}\put(46.8,45.2){\line(1,0){39.1}}
\put(48.5,44.8){\line(1,0){38.2}}\put(50.4,44.4){\line(1,0){37.0}}
\put(52.3,44.0){\line(1,0){35.8}}\put(55.0,43.6){\line(1,0){33.8}}
\put(58.3,43.2){\line(1,0){31.2}}\put(61.7,42.8){\line(1,0){28.6}}
\put(65.3,42.4){\line(1,0){25.6}}\put(68.0,42.0){\line(1,0){23.5}}
\put(70.2,41.6){\line(1,0){22.3}}\put(72.2,41.2){\line(1,0){20.9}}
\put(74.2,40.8){\line(1,0){19.4}}\put(75.9,40.4){\line(1,0){18.3}}
\put(77.5,40.0){\line(1,0){17.3}}\put(79.0,39.6){\line(1,0){16.3}}
\put(80.8,39.2){\line(1,0){15.3}}\put(82.3,38.8){\line(1,0){14.4}}
\put(84.0,38.4){\line(1,0){13.3}}\put(85.5,38.0){\line(1,0){12.6}}
\put(87,37.6){\line(1,0){11.7}}\put(88.5,37.2){\line(1,0){10.8}}
\put(90.0,36.8){\line(1,0){9.9}}\put(91,36.4){\line(1,0){9.3}}
\put(92.2,36.0){\line(1,0){8.8}}\put(93.4,35.6){\line(1,0){8.2}}
\put(94.4,35.2){\line(1,0){7.4}}\put(95.8,34.8){\line(1,0){6.7}}
\put(97.1,34.4){\line(1,0){5.6}}\put(98.3,34.0){\line(1,0){4.6}}
\put(99.5,33.6){\line(1,0){3.5}}\put(100.3,33.2){\line(1,0){3.0}}
\color{black} \thinlines 
\qbezier(40,50)(40,60)(63,54) \qbezier(40,50)(40,45)(60,43)
\qbezier(63,54)(86,47)(98,38)
\qbezier(60,43)(80,40.5)(97.2,34.3)
\qbezier(98,38)(109,30)(97.2,34.3) \put(47,48){\small$P_{-1}$}
%
 \thicklines \color{lightgrey}
\put(40.2,130){\line(1,0){35}} \put(40.2,130.4){\line(1,0){34}}
\put(40.2,130.8){\line(1,0){33}}
\put(40.4,131.2){\line(1,0){31.6}}
\put(40.4,131.6){\line(1,0){30.4}}
\put(40.5,132){\line(1,0){29.3}}
\put(40.5,132.4){\line(1,0){28.1}}\put(40.8,132.8){\line(1,0){26.5}}
\put(41.0,133.2){\line(1,0){25.0}}\put(41.5,133.6){\line(1,0){23.3}}
\put(41.8,134.0){\line(1,0){21.3}}\put(42.1,134.4){\line(1,0){19.3}}
\put(42.5,134.8){\line(1,0){17.3}}\put(43.2,135.2){\line(1,0){14.8}}
\put(44.3,135.6){\line(1,0){11.8}}\put(46.3,136.0){\line(1,0){6.8}}
\put(40.2,129.6){\line(1,0){36.1}}\put(40.4,129.2){\line(1,0){36.7}}
\put(40.6,128.8){\line(1,0){37.4}}\put(40.8,128.4){\line(1,0){38.2}}
\put(41.1,128.0){\line(1,0){38.9}}\put(41.5,127.6){\line(1,0){39.4}}
\put(42.1,127.2){\line(1,0){39.5}}\put(42.7,126.8){\line(1,0){39.8}}
\put(43.7,126.4){\line(1,0){39.7}}\put(44.5,126){\line(1,0){39.6}}
\put(45.7,125.6){\line(1,0){39.3}}\put(46.8,125.2){\line(1,0){39.1}}
\put(48.5,124.8){\line(1,0){38.2}}\put(50.4,124.4){\line(1,0){37.0}}
\put(52.3,124.0){\line(1,0){35.8}}\put(55.0,123.6){\line(1,0){33.8}}
\put(58.3,123.2){\line(1,0){31.2}}\put(61.7,122.8){\line(1,0){28.6}}
\put(65.3,122.4){\line(1,0){25.6}}\put(68.0,122.0){\line(1,0){23.5}}
\put(70.2,121.6){\line(1,0){21.9}}\put(72.2,121.2){\line(1,0){20.9}}
\put(74.2,120.8){\line(1,0){19.4}}\put(75.9,120.4){\line(1,0){18.3}}
\put(77.5,120.0){\line(1,0){17.3}}\put(79.0,119.6){\line(1,0){16.3}}
\put(80.8,119.2){\line(1,0){15.3}}\put(82.3,118.8){\line(1,0){14.4}}
\put(84.0,118.4){\line(1,0){13.3}}\put(85.5,118.0){\line(1,0){12.6}}
\put(87,117.6){\line(1,0){11.7}}\put(88.5,117.2){\line(1,0){10.8}}
\put(90.0,116.8){\line(1,0){9.9}}\put(91,116.4){\line(1,0){9.3}}
\put(92.2,116.0){\line(1,0){8.8}}\put(93.4,115.6){\line(1,0){8.2}}
\put(94.4,115.2){\line(1,0){7.4}}\put(95.8,114.8){\line(1,0){6.7}}
\put(97.1,114.4){\line(1,0){5.6}}\put(98.3,114.0){\line(1,0){4.6}}
\put(99.5,113.6){\line(1,0){3.5}}\put(100.3,113.2){\line(1,0){3.0}}
\color{black} \thinlines 
\qbezier(40,130)(40,140)(63,134)
\qbezier(40,130)(40,125)(60,123)
\qbezier(63,134)(86,127)(98,118)
\qbezier(60,123)(80,120.5)(97.2,114.3)
\qbezier(98,118)(109,110)(97.2,114.3) \put(47,128){\small $P_1$}
  \thicklines 
  \put(20,69){\line(0,1){40}} \put(130.2,40){\line(0,1){40}}
  \qbezier(20,69)(70,76)(130,40) \qbezier(20,109)(70,116)(130,80)
  \thinlines
  \put(6,89){\small $\Gammai$} \put(113.5,67){\small $\Gammao$}
  \put(75,107.5){\small $\Gammap$} \put(75,67.5){\small $\Gammam$}
  \put(37,77){\small $\Gammaw$}
  \put(107,79){\small $\Omega$} \put(107,135){\small $\cO$}
  \dashline[+28]{2}(130.1,70)(130.1,108)
  \dashline[+28]{2}(130.1,124)(130.1,152)
  \dashline[+28]{2}(130.1,40)(130.1,30)
  \put(8,67){\small $A_0$} \put(8,107){\small $A_1$}
  \put(133,38){\small $B_0$} \put(133,78){\small $B_1$}
  \put(118,114.5){\small $x_1=d$}
  \put(6.5,134){\small $\gammai$} \put(134,138){\small $\gammao$}
  \thicklines
  \put(70,28){$\vdots$} \put(70,139){$\vdots$}
  \put(19,9){\small Fig.~1: \ The profile cascade with}
  \put(42,0){\small marked one spatial period}
  \end{picture}
\end{wrapfigure}

\noindent
the whole profile cascade to the flow through just one spatial
period, which is denoted by $\Omega$, see Fig.~1. This approach
is used e.g.~in papers \cite{FeNe1}--\cite{FeNe3} and
\cite{TNe1}--\cite{TNe3}, where the qualitative analysis of
corresponding mathematical models is studied, and in papers
\cite{DFF}, \cite{KLP}, \cite{SPKF}, devoted to the numerical
analysis of the models or corresponding numerical calculations.

\paragraph{Classical formulation of the problem in one spatial
period.} We assume that $\Omega$ is a Lipschitzian sub--domain
of $\R^2_{(0,d)}$, such that its boundary consists of the line
segment $\Gammai\equiv A_0A_1$ of length $\tau$, the line
segment $\Gammao\equiv B_0B_1$ of the same length $\tau$, the
closed curve $\Gammaw$ (the boundary of profile $P_0$) and the
curves $\Gammam$, $\Gammap$ such that
$\Gammap=\Gammam+\tau\br\bfe_2$. (See Fig.~1.) As the curves
$\Gammam$ and $\Gammap$ in fact represent artificial boundaries
of $\Omega$, chosen in $\cO$, we may assume without loss of
generality that both $\Gammam$ and $\Gammap$ are of the class
$C^2$.

The reduced mathematical problem consists of the equations
\begin{align}
\partial_t\bfu-\nu\Delta\bfu+\bfu\cdot\nabla\bfu+\nabla p\ &=\ \bff
\label{1.1} \\
\div\bfu\ &=\ 0 \label{1.2}
\end{align}
in the space--time cylinder $\Omega\times(0,T)$ (where $T>0$),
completed by appropriate initial and boundary conditions. Here,
$\bfu=(u_1,u_2)$ denotes the unknown velocity of the moving
fluid, $p$ denotes the unknown pressure, positive constant
$\nu$ is the kinematic coefficient of viscosity and $\bff$ is
the external body force. The density of the fluid (which is
also supposed to be a positive constant) can be without loss of
generality supposed to be equal to one. Equation (\ref{1.1})
(the Navier--Stokes equation) expresses the conservation of
momentum and equation (\ref{1.2}) (the equation of continuity)
expresses the conservation of mass.

We assume that the fluid flows into the cascade through the
straight line $\gammai$ (the $x_2$--axis) and essentially
leaves the cascade through the straight line $\gammao$, whose
equation is $x_1=d$. (By ``essentially'' we mean that possible
reverse flows on $\gammao$ are not excluded.) This is why we
complete equations (\ref{1.1}), (\ref{1.2}) by the
inhomogeneous Dirichlet boundary condition
\begin{equation}
\bfu\ =\ \bfg \qquad \mbox{on}\ \Gammai, \label{1.3}
\end{equation}
the homogeneous Dirichlet boundary condition
\begin{equation}
\bfu\ =\ \bfzero \qquad \mbox{on}\ \Gammaw \label{1.4}
\end{equation}
and appropriate conditions on $\Gammam$, $\Gammap$ and
$\Gammao$. Due to the assumed spatial periodicity of the flow,
it is reasonable to prescribe the boundary conditions of
periodicity on $\Gammam$ and $\Gammap$:
\begin{align}
\bfu(x_1,x_2+\tau)\ &=\ \bfu(x_1,x_2) && \mbox{for}\
\bfx\equiv(x_1,x_2)\in\Gammam, \label{1.5} \\ \noalign{\vskip 4pt}
\frac{\partial\bfu}{\partial\bfn}(x_1,x_2+\tau)\ &=\
-\frac{\partial\bfu}{\partial\bfn}(x_1,x_2) &&
\mbox{for}\ \bfx\equiv(x_1,x_2)\in\Gammam, \label{1.6} \\
\noalign{\vskip 4pt}
p(x_1,x_2+\tau)\ &=\ p(x_1,x_2) && \mbox{for}\
\bfx\equiv(x_1,x_2)\in\Gammam. \label{1.7}
\end{align}
On $\Gammao$, various authors use various artificial boundary
conditions. One of the most popular ones is the condition
\begin{equation}
-\nu\, \frac{\partial\bfu}{\partial\bfn}+p\br\bfn\ =\ \bfh,
\label{1.8}
\end{equation}
\vspace{2pt} \noindent
where $\bfh$ is a given vector--function on $\Gammao$ and
$\bfn$ denotes the unit outward normal vector, which is equal
to $\bfe_1$ on $\Gammao$. The boundary condition (\ref{1.8})
(with $\bfh=\bfzero$) is often called the ``do nothing''
condition, because it naturally follows from a weak formulation
of the boundary--value problem, see e.g.~\cite{Glow} and
\cite{HeRaTu}.

\paragraph{On some previous related results.} Since this
condition does not enable one to control the amount of kinetic
energy in $\Omega$ in the case of an possible backward flow on
$\Gammao$, many authors also use various modifications of
condition (\ref{1.8}). (See e.g.~\cite{BrFa}, \cite{FeNe1},
\cite{FeNe2}, \cite{FeNe3}, \cite{TNe1}, \cite{TNe2}.) The
modified conditions enable one to derive a priori estimates of
solutions and existence of weak solutions. In paper
\cite{FeNe1}, the existence of a steady weak solution of the
problem (\ref{1.1})--(\ref{1.8}) was proven for ``sufficiently
small'' velocity profile $\bfg$ on $\Gammai$, while in
\cite{FeNe3} and \cite{TNe3}, the function $\bfg$ can be
arbitrarily large. The existence of a non-steady weak solution
on an arbitrarily long time interval has been proven in
\cite{FeNe2}. In papers \cite{KuSka} and \cite{Ku}, the authors
use the boundary condition (\ref{1.8}) on an ``outflow'' part
of the boundary for a flow in a channel, and they prove the
existence of a weak solution for ``small data''. Possible
backward flows on the ``outflow'' of the channel are controlled
by means of additional conditions in \cite{KraNe1},
\cite{KraNe2}, \cite{KraNe3}, which consequently cause that the
Navier--Stokes equations must be replaced by the Navier--Stokes
variational inequalities.

There are no results in literature about the regularity up to
the boundary of existing weak solutions. The question of higher
regularity of a solution is closely connected with the so
called {\it maximum regularity property} of the associated
steady Stokes problem, which we obtain from the Navier--Stokes
problem if we neglect the derivative with respect to $t$ and
the nonlinear term. It consists of the equations
\begin{equation}
-\nu\Delta\bfu+\nabla p\ =\ \bff \label{1.11}
\end{equation}
and (\ref{1.2}) (in $\Omega$), and the boundary conditions
(\ref{1.3})--(\ref{1.8}). The maximum regularity property
roughly speaking means that the solution $\bfu$, respectively
$p$, has by two, respectively one, spatial derivatives more
than function $\bff$, and the derivatives are integrable with
the same power as $\bff$. (See Theorem \ref{T2}.) An analogous
property of the steady Stokes problem is mostly known only in
the case of a smooth domain $\Omega$, see e.g.~\cite[Theorem
I.2.2]{Te}, \cite[Theorem III.3]{La}, \cite[Theorem IV.6.1]{Ga}
and \cite[Theorem III.2.1.1]{So} for the Stokes problem with
the inhomogeneous Dirichlet boundary condition, \cite{AlBAmEs},
\cite{ChOsQi} for problems with the Navier--type boundary
condition, \cite{AmEsGh}, \cite{ChQi} for problems with
Navier's boundary condition, \cite{Me1} for the 2D Stokes
problem with the Neumann boundary condition (i.e.~prescribing
the normal part of the stress tensor on the boundary) and
\cite{Me2} for the 2D Stokes problem, prescribing the normal
component of velocity and the pressure on the boundary.
Concerning the maximum regularity property of the Stokes
problem in non--smooth domains, we can cite \cite{Gr2},
\cite{KeOs} and \cite{Dau}, where the authors considered the
Stokes problem in a 2D polygonal domain with the Dirichlet
boundary condition. In paper \cite{KuBe}, the authors studied
the Stokes problem in a 2D channel $D$ of a special geometry,
considering the homogeneous Dirichlet boundary condition on the
walls and the homogeneous condition (\ref{1.8}) on the outflow,
and proved that the velocity is in $\bfW^{2-\beta,2}(D)$ for
certain $\beta\in(0,1)$, provided that $\bff\in\bfL^2(D)$. (See
\cite[Theorem 2.1]{KuBe}.)

\paragraph{On results of this paper.} \ In this paper, we at first
verify the existence of a weak solution $\bfu$ to the Stokes
problem (\ref{1.11}), (\ref{1.2})--(\ref{1.5}) and we show that
an appropriate pressure $p$ can be chosen so that the pair
$(\bfu,p)$ satisfies equations (\ref{1.11}), (\ref{1.2}) in the
sense of distributions in $\Omega$ and the boundary condition
(\ref{1.8}) as an equality in $\bfW^{-1/2,2}(\Gammao)$. The
boundary conditions (\ref{1.3})--(\ref{1.5}) are satisfied in
the usual sense of traces. (Theorem \ref{T1}.) Then, for ``smooth''
input data, we prove the existence of a strong solution of the
Stokes problem (\ref{1.11}), (\ref{1.2})--(\ref{1.5}) and its
maximum regularity property. (Theorem \ref{T2}.) This result cannot be
simply deduced from the previous aforementioned papers, because
our domain $\Omega$ is not smooth and we consider altogether
three types of boundary conditions, two of whose ``meet'' at
the corner points $A_0$, $A_1$, $B_0$ and $B_1$ of domain
$\Omega$. In order to prove the regularity ``up to the
boundary'' in the neighborhood of $\Gammao$ and $\Gammaw$, we
use the fact that the solution satisfies Dirichlet--type
boundary conditions on $\Gammai$ and $\Gammaw$ and we apply
known results on the Stokes problem with Dirichlet's boundary
condition. In the neighborhood of $\Gammao$, we use the fact
that $\Gammao$ is a part of a straight line and we apply the
technique of the so called difference quotients, whose
originality is usually attributed to L.~Nirenberg and which is
described e.g.~in \cite{Ag} and \cite{Gr1}. As to the curves
$\Gammam$ and $\Gammap$, we use the possibility of an
appropriate extension of a solution in the $x_2$--direction,
which enables us to avoid problems in neighborhoods of the
corner points $A_0$, $A_1$, $B_0$, $B_1$ and to study the
regularity in the neighborhood of $\Gammam$ and $\Gammap$ as an
interior problem.

\section{The weak steady Stokes problem} \label{S2}

\paragraph{Notation.} Recall that $\Omega$ is a domain in $\R^2$,
sketched on Fig.~1. Its boundary consists of the curves
$\Gammai$, $\Gammao$, $\Gammam$, $\Gammap$ and $\Gammaw$,
described in Section \ref{S1}. We denote by $\bfn=(n_1,n_2)$
the outer normal vector field on $\partial\Omega$. Note that
$\bfn=-\bfe_1$ on $\Gammai$ and $\bfn=\bfe_1$ on $\Gammao$.

\begin{list}{$\circ$}
{\setlength{\topsep 0.5mm}
\setlength{\itemsep 1.0mm}
\setlength{\leftmargin 10pt}
\setlength{\rightmargin 0pt}
\setlength{\labelwidth 6pt}}

\item
$\Gammai^0$ , respectively $\Gammao^0$, denotes the open line
segment without the end points $A_0,\, A_1$, respectively
$B_0,\, B_1$. Similarly, $\Gammam^0$, respectively $\Gammap^0$
denotes the curve $\Gammam$, respectively $\Gammap$, without
the end points $A_0$, $B_0$, respectively $A_1$, $B_1$.

\item
We denote vector functions and spaces of vector functions by
boldface letters. Tensor functions are denoted e.g.~by $\bbF$
or $\bbG$ and spaces of tensor functions are marked by a superscript
$2\times 2$.

\item
We denote by $\|\, .\, \|_r$ the norm in $L^r(\Omega)$ or in
$\bfL^r(\Omega)$ or in $L^r(\Omega)^{2\times 2}$. Similarly,
$\|\, .\, \|_{r,s}$ is the norm in $W^{r,s}(\Omega)$ or in
$\bfW^{r,s}(\Omega)$ or in $W^{r,s}(\Omega)^{2\times 2}$. The
scalar product in $L^2(\Omega)$ or in $\bfL^2(\Omega)$ or in
$L^2(\Omega)^{2\times 2}$ is denoted by $(\, .\, ,\, .\, )_2$.

\item
$W^{-1/2,2}(\Gammao)$ is the dual space to
$W^{1/2,2}(\Gammao)$. Note that the spaces
$W^{-s,2}(.\br.\br.)$ (for $s>0$) are usually defined to be the
dual spaces to $W^{s,2}_0(.\br.\br.)$, see
e.g.~\cite[Definition I.12.1]{LiMa}. However, as
$W^{1/2,2}(\Gammao)=W^{1/2,2}_0(\Gammao)$ (see \cite[Theorem
II.11.1]{LiMa}), it plays no role whether we define
$W^{-1/2,2}(\Gammao)$ to be the dual to $W^{1/2,2}(\Gammao)$ or
$W^{1/2,2}_0(\Gammao)$.

\item
$\cCs$ denotes the linear space of infinitely differentiable
divergence--free vector functions in $\overline{\Omega}$, whose
support is disjoint with $\Gammai\cup\Gammaw$ and that satisfy,
together with all their derivatives (of all orders), the
condition of periodicity (\ref{1.5}). Note that each
$\bfw\in\cCs$ satisfies $\int_{\Gammao}\bfw\cdot\bfn\; \rmd
l=0$.

\item
$\Vs$ is the closure of $\cCs$ in $\bfW^{1,2}(\Omega)$. The
space $\Vs$ can be characterized as a space of divergence--free
vector functions $\bfv\in\bfW^{1,2}(\Omega)$, whose traces on
$\Gammai\cup\Gammaw$ are equal to zero, the traces on $\Gammam$
and $\Gammap$ satisfy the condition of periodicity (\ref{1.5})
and the traces on $\Gammao$ satisfy
$\int_{\Gammao}\bfv\cdot\bfn\; \rmd l=0$. Note that as
functions from $\Vs$ are equal to zero on $\Gammai\cup\Gammaw$
(in the sense of traces) and domain $\Omega$ is bounded, the
norm in $\Vs$ is equivalent to $\|\nabla.\, \|_2$.

\item
We denote by $\bfW^{-1,2}_0(\Omega)$ the dual space to
$\bfW^{1,2}(\Omega)$, by $\bfW^{-1,2}(\Omega)$ the dual space
to $\bfW^{1,2}_0(\Omega)$, and by $\|\, .\, \|_{\bfW^{-1,2}_0}$
and $\|\, .\, \|_{\bfW^{-1,2}}$ the corresponding norms.

\item
$\Vsd$ is the dual space to $\Vs$. The duality pairing between
$\Vsd$ and $\Vs$ is denoted by $\langle\, .\, ,\, .\,
\rangle_{\sigma}$. The norm in $\Vsd$ is denoted by $\|\, .\,
\|_{\vsd}$.

\item
Denote by $\cA$ the linear mapping of $\Vs$ to $\Vsd$, defined
by the equation
\begin{displaymath}
\blangle\cA\bfv,\bfw\brangle_{\sigma}\ :=\ \int_{\Omega}
\nabla\bfv:\nabla\bfw\; \rmd\bfx \qquad \mbox{for}\
\bfv,\bfw\in\Vs.
\end{displaymath}

\vspace{-1.5mm} \item
$c$ denotes a generic constant, i.e.~a constant whose values
may change throughout the text.

\end{list}

\begin{lemma} \label{L2.4}
Operator $\cA$ is a one--to--one closed bounded operator from
$\Vs$ to the dual space $\Vsd$ with the domain $D(\cA)=\Vs$ and
range $R(\cA)=\Vsd$. The inverse operator $\cA^{-1}$ is
bounded, as an operator from $\Vsd$ to $\Vs$.
\end{lemma}

\begin{proof}
Denote by $N(\cA)$ the null space of $\cA$. Let $\bfv\in
N(\cA)$. Then
\begin{displaymath}
\blangle\cA\bfv,\bfw\brangle_{\sigma}\ =\
(\nabla\bfv,\nabla\bfw)_2\ =\ 0
\end{displaymath}
for all $\bfw\in\Vs$. The choice $\bfw=\bfv$ yields
$(\nabla\bfv,\nabla\bfv)_2=\|\nabla\bfv\|_2^2=0$. This
(together with the boundary conditions on $\Gammai\cup\Gammaw$)
implies that $\bfv=\bfzero$. Thus, operator $\cA$ is injective.

The boundedness of $\cA$ can be proven in this way::
\begin{align*}
\|\cA\bfv\|_{\vsd}\ &= \sup_{\atop{\bfw\in\Vs,}
{\bfw\not=\bfzero}}
\frac{|\br\langle
\cA\bfv,\bfw\rangle_{\sigma}\br|}{\|\bfw\|_{1,2}}\ =
\sup_{\atop{\bfw\in\Vs,}{\bfw\not=\bfzero}}
\frac{|(\nabla\bfv,\,
\nabla\bfw)_2|} {\|\bfw\|_{1,2}}\
\leq\ c\, \|\nabla\bfv\|_2.
\end{align*}

The equality $D(\cA)=\Vs$ follows from the definition of $\cA$.
The equality $R(\cA)=\Vsd$ follows from Riesz' theorem and the
equivalence of the scalar products $(\, .\, ,\, .\, )_{1,2}$
and $(\nabla.\, ,\nabla.\, )_2$ in $\Vs$: if $\bff\in\Vsd$ then
there exists $\bfv\in\Vs$ such that
$\langle\bff,\bfw\rangle_{\sigma}=(\nabla\bfv,\nabla\bfw)_2$
for all $\bfw\in\Vs$. Hence $\bff=\cA\bfv$.

Operator $\cA$ is closed, as a bounded linear operator, defined
on the whole space $\Vs$. Hence the inverse operator $\cA^{-1}$
is also closed. As a closed linear operator, defined on the
whole space $\Vsd$, $\cA^{-1}$ is bounded from $\Vsd$ to $\Vs$.
\end{proof}

\vspace{4pt}
Assume that $\bbF\in L^2(\Omega)^{2\times 2}$. Define a bounded
linear functional $\bfF\in\Vsd$ by the formula
\begin{equation}
\blangle\bfF,\bfw\brangle_{\sigma}\ :=\
-\int_{\Omega}\bbF:\nabla\bfw\; \rmd\bfx \qquad \mbox{for all}\
\bfw\in\Vs. \label{2.1}
\end{equation}

The next lemma comes from \cite[Sec.~3]{FeNe3}.

\begin{lemma} \label{L2.5}
Assume that $\bfg\in\bfW^{1/2,2}(\Gammai)$ is a given function
on $\Gammai$, such that it can be extended from $\Gammai$ to
$\gammai$ as a function $\tau$--periodic function from
$\bfW^{1/2,2}_{loc}(\gammai)$. Then there exists a
divergence--free extension $\bfg_*$ of $\bfg$ from $\Gammai$ to
$\Omega$, such that $\bfg_*\in\bfW^{1,2}(\Omega)$,

\vspace{4pt}
a) \ $\|\bfg_*\|_{1,2}\leq c\, \|\bfg\|_{1/2,2;\, \Gammai}$
(where $c$ is independent of $\bfg$ and $\bfgs$),

\vspace{4pt}
b) \ $\bfg_*$ satisfies the condition of periodicity
(\ref{1.5}) on $\Gammam\cup\Gammap$,

\vspace{4pt}
c) \ $\bfg_*=(\Phi/\tau)\, \bfe_1$ \ in a neighborhood of
$\Gammao$, where $\Phi=\int_{\Gammai}\bfg\cdot\bfe_1\; \rmd l$,

\vspace{4pt}
d) \ $\bfg_*=\bfzero$ on $\Gammaw$ in the sense of traces.
\end{lemma}

\noindent
Note that the function $\bfg$ is assumed to be in
$\bfW^{s,2}(\Gammai)$ for some $s>\frac{1}{2}$ and to satisfy
the condition $\bfg(A_0)=\bfg(A_1)$ in \cite{FeNe3}. However,
this condition can be simply replaced by our assumption,
i.e.~that $\bfg\in\bfW^{1/2,2}(\Gammai)$ has a $\tau$--periodic
extension in $\bfW^{1/2,2}_{loc}(\gammai)$, with no affect on
the proof in \cite{FeNe3}.

Let $\bfg_*$ be the function, provided by Lemma \ref{L2.5}.
Define a bounded linear functional $\bfG$ on $\Vs$ by the
formula
\vspace{-4pt}
\begin{equation}
\blangle\bfG,\bfw\brangle_{\sigma}\ :=\
-\int_{\Omega}\nabla\bfg_*:\nabla\bfw\; \rmd\bfx \qquad \mbox{for
all}\ \bfw\in\Vs. \label{2.34}
\end{equation}

\begin{theorem} [on a weak solution of the Stokes problem
(\ref{1.2})--(\ref{1.5}), (\ref{1.8}), (\ref{1.11}))]
\label{T1} Let $\bfF,\, \bfG\in\Vsd$ be defined by formulas
(\ref{2.1}) and (\ref{2.34}), respectively. Then the equation
$\nu\cA\bfv=\bfF+\nu\br\bfG$ has a unique solution
$\bfv\in\Vs$. Moreover, there exists $p\in L^2(\Omega)$ such
that the functions $\bfu:=\bfg_*+\bfv$, which is
divergence--free, and $p$ satisfy the equation
\begin{equation}
-\nu\Delta\bfu+\nabla p\ =\ \div\bbF \label{2.33}
\end{equation}
in the sense of distributions in $\Omega$, the boundary
condition
\begin{equation}
\bigl(-\nu\br\nabla\bfu+p\br\bbI-\bbF\bigr)\cdot\bfn\ =\ \bfzero
\label{2.31}
\end{equation}
as an equality in $\bfW^{-1/2,2}(\Gammao)$ and the estimate
\phantom{$\cn03$}
\begin{equation}
\|p\|_2\ \leq\ \cc03\, \bigl( \|\nabla\bfu\|_2+ \|\bbF\|_2\bigr),
\label{2.29}
\end{equation}
where $\cc03=\cc03(\Omega,\nu)$.
\end{theorem}

\begin{proof}
The existence and uniqueness of the solution $\bfv$ of the
equation $\nu\cA\bfv=\bfF+\nu\br\bfG$ follows from Lemma
\ref{L2.4}.

Denote by $\cCn$ the space of all infinitely differentiable
vector functions in $\Omega$ with a compact support in $\Omega$
and put $\cCns:=\cCs\cap\cCn$. Suppose at first that
$\bfw\in\cCns$. Then
$\langle\cA\bfv,\bfw\rangle_{\sigma}=-\llangle\Delta\bfv,\bfw
\rrangle$ for all $\bfw\in\cCns$, where $\llangle\, .\, ,\, .\,
\rrangle$ denotes the pairing between a distribution in
$\Omega$ and a function from $\cCn$. Similarly, we can write
$\langle\bfF,\bfw\rangle_{\sigma}=\llangle\div\bbF,\bfw\rrangle$
and $\langle\bfG,\bfw\rangle_{\sigma}=
\llangle\Delta\bfg_*,\bfw\rrangle$. Then the equation
$\langle\nu\cA\bfv-\bfF-\nu\, \bfG,\bfw\rangle_{\sigma}=0$ and
the identity $\bfu=\bfg_*+\bfv$ imply that
\begin{displaymath}
\llangle\nu\Delta\bfu+\div\bbF,\bfw\rrangle\ =\ 0 \qquad
\mbox{for all}\ \bfw\in\cCns.
\end{displaymath}
Thus, $\nu\Delta\bfu+\div\bbF$ is a distribution that vanishes
on all divergence--free functions $\bfw\in\cCn$. Due De Rham's
lemma (see \cite[p.~14]{Te}), there exists a distribution $p_0$
in $\Omega$, such that
\begin{equation}
\nu\Delta\bfu+\div\bbF\ =\ \nabla p_0 \label{2.11}
\end{equation}
holds in $\Omega$ in the sense of distributions. As both
$\nu\Delta\bfu$ and $\div\bbF$ can also be naturally identified
with bounded linear functionals on $\bfW^{1,2}_0(\Omega)$,
i.e.~elements of the dual space $\bfW^{-1,2}(\Omega)$, $\nabla
p_0$ belongs to $\bfW^{-1,2}(\Omega)$, too. Applying
\cite[Proposition I.1.2]{Te}, we deduce that $p_0\in
L^2(\Omega)$, it can be chosen so that $\int_{\Omega}p_0\;
\rmd\bfx=0$, and
\vspace{-2pt}
\begin{equation}
\|p_0\|_2\ \leq\ c\, \bigl\| \nu\Delta\bfu+\div\bbF
\bigr\|_{\bfW^{-1,2}}, \label{2.30}
\end{equation}

\vspace{-2pt} \noindent
where $c$ depends only on $\nu$ and $\Omega$.

Since $-\nu\, \nabla\bfu+p_0\br\bbI-\bbF\in
L^2(\Omega)^{2\times 2}$ and $\div(-\nu\,
\nabla\bfu+p_0\br\bbI-\bbF)=\bfzero$, we can apply
\cite[Theorem III.2.2]{Ga} and deduce that $(-\nu\,
\nabla\bfu+p_0\br\bbI-\bbF)\cdot\bfn\in
\bfW^{-1/2,2}(\Gammao)$.

Let $\bfw\in\cCs$. The equations (\ref{2.11}) and
$\nu\cA\bfv=\bfF+\nu\br\bfG$, the formula $\bfu=\bfv+\bfgs$ and
the generalized Gauss identity (see \cite[p.~160]{Ga}) imply
that
\begin{align*}
0\ &=\ \bigl(\div[\nu\nabla\bfu+\bbF-p_0\br\bbI\br],\bfw\bigr)_2 \\
&=\ \blangle (\nu\br\nabla\bfu+\bbF-p_0\br\bbI)\cdot\bfn,\bfw
\brangle_{\Gammao}-\int_{\Omega}[\nu\br\nabla
\bfu+\bbF]:\nabla\bfw\; \rmd\bfx-\int_{\Omega}p_0\, \div\bfw\;
\rmd\bfx \\
&=\ \blangle (\nu\br\nabla\bfu+\bbF-p_0\br\bbI)\cdot\bfn,\bfw
\brangle_{\Gammao}-\blangle\nu\cA\bfv,\bfw\brangle_{\sigma}+
\blangle\nu\bfG,\bfw\brangle_{\sigma}+\blangle\bfF,
\bfw\brangle_{\sigma} \\
&=\ \blangle (\nu\br\nabla\bfu+\bbF-p_0\br\bbI)\cdot\bfn,\bfw
\brangle_{\Gammao}.
\end{align*}
It can be deduced e.g.~from \cite[Sec.~3]{FeNe3} that the set
of traces of all functions from $\cCs$ on $\Gammao$ is dense in
the set of all functions $\bfw\in\bfW^{1/2,2}(\Gammao)$, such
that $\int_{\Gammao}\bfw\cdot\bfn\; \rmd l=0$. Hence there
exists $\cn04\in\R$ such that $\bfu$ and $p$ satisfy
\vspace{-2pt}
\begin{displaymath}
\bigl(\nu\br\nabla\bfu-p_0\br\bbI+\bbF\bigr)\cdot\bfn\ =\
\cc04\br\bfn,
\end{displaymath}

\vspace{-2pt} \noindent
as an equality in $\bfW^{-1/2,2}(\Gammao)$. Put $p:=p_0+\cc04$.
Then $\bfu$ and $p$ satisfy the boundary condition (\ref{2.31})
as an equality in $\bfW^{-1/2,2}(\Gammao)$. It follows from
(\ref{2.11}) that $\bfu$, $p$ satisfy equation (\ref{2.33}) in
the sense of distributions in $\Omega$. Finally, (\ref{2.30})
implies that estimate (\ref{2.29}) holds, too.
\end{proof}

Function $\bfu$ represents a weak solution of the Stokes
problem (\ref{1.11}), (\ref{1.2})--(\ref{1.5}), where
$\bff=\div\bbF$, with the boundary condition (\ref{2.31}) on
$\Gammao$.

The next lemma follows from \cite[Theorem 2.5]{GeHeHi}. It
shows that it is not a loss of generality if we write the right
hand side of equation (\ref{2.33}) in the form $\div\bbF$
instead of just $\bff$. On the other hand, considering the
right hand side of (\ref{2.33}) in the form $\div\bbF$ enables
us to deduce that $\bfv$ and $p$ satisfy (\ref{2.31}), as an
equality in $\bfW^{-1/2,2}(\Gammao)$. An analogue, having just
$\bff\in\bfW^{-1,2}(\Omega)$ instead of $\div\bbF$, would not
be possible.

\begin{lemma} \label{L2.3}
Let $\bff\in\bfW^{-1,2}_0(\Omega)$. Then there exists $\bbF\in
L^2(\Omega)^{2\times 2}$, satisfying $\div\bbF=\bff$ in the
sense of distributions in $\Omega$ and
\begin{equation}
\|\bbF\|_2\ \leq\ c\, \|\bff\|_{\bfW^{-1,2}_0}\, , \label{2.2}
\end{equation}
where $c$ is independent of $\bff$ and $\bbF$.
\end{lemma}


\section{The strong steady Stokes problem} \label{S3}

\paragraph{Notation.} \ In section \ref{S3}, we also use this
notation:

\begin{list}{$\circ$}
{\setlength{\topsep 0pt}
\setlength{\itemsep 2pt}
\setlength{\leftmargin 10pt}
\setlength{\rightmargin 0pt}
\setlength{\labelwidth 6pt}}

\item
$\Ls$ is the closure of $\cCs$ in $\bfL^2(\Omega)$. Functions
from $\Ls$ are divergence--free in the sense of distributions
in $\Omega$ and their normal components (in the sense of
traces) belong to the space $W^{-1/2,2}(\partial\Omega)$ (the
dual to $W^{1/2,2}(\partial\Omega)$, see \cite[Theorem
III.2.2]{Ga}). Moreover, $\bfv\cdot\bfn=\bfzero$ holds as an
equality in $W^{-1/2,2}(\Gammai\cup\Gammaw)$,
\begin{displaymath}
    \bfv(x_1,x_2+\tau)\cdot\bfn(x_1,x_2+\tau)\ =\
    -\bfv(x_1,x_2)\cdot\bfn(x_1,x_2)
\end{displaymath}
holds as an equality in $W^{-1/2,2}(\Gammam)$ and $\langle
\bfv\cdot\bfn,\br 1\rangle_{\Gammao}=0$, where $\langle\, .\,
,\, .\, \rangle_{\Gammao}$ denotes the duality pairing between
$W^{-1/2,2}(\Gammao)$ and $W^{1/2,2}(\Gammao)$.

\item
$\R^2_{d-}$ denotes the half-plane $x_1<d$. Recall that
$\R^2_{(0,d)}:=\{\bfx\equiv(x_1,x_2)\in\R^2;\ 0<x_1<d\}$ and
$\cO=\R^2_{(0,d)}\smallsetminus\cup_{k=-\infty}^{\infty}P_k$.

\item
$W^{k,2}_{\rm per}(\cO)$ (for $k\in\N$) denotes the space of
functions from $W^{k,2}_{loc}(\cO)$, $\tau$--periodic in
variable $x_2$.

\item
$W^{k,2}_{\rm per}(\Omega)$ is the space of functions, that can
be extended from $\Omega$ to $\cO$ as functions in
$W^{k,2}_{\rm per}(\cO)$. (Obviously, the traces of these
functions on $\Gammam$ and $\Gammap$ satisfy the condition of
periodicity, analogous to (\ref{1.5}).)

\item
$W^{k-1/2,2}_{per}(\gammao)$ (for $k\in\N$) is the space of
$\tau$--periodic functions in $W^{k-1/2,2}_{loc}(\gammao)$.

\item
$W^{k-1/2,2}_{\rm per}(\Gammao)$ is the space of functions from
$W^{k-1/2,2}(\Gammao)$, that can be extended from $\Gammao$ to
$\gammao$ as functions in $W^{k-1/2,2}_{\rm per}(\gammao)$. The
space $W^{k-1/2,2}_{\rm per}(\Gammai)$ is defined by analogy.

\item
Spaces of corresponding vector functions are again denoted by
boldface letters and spaces of corresponding tensor functions
are marked by the superscript $2\times 2$.

\end{list}

\begin{lemma} \label{L2.9}
Let $n\in\N$ and the functions $g_0\in W^{n+1/2,2}_{\rm
per}(\Gammao)$, $g_1\in W^{n-1/2,2}_{\rm per}(\Gammao)$,
$\dots,$ $g_n\in W^{1/2,2}_{\rm per}(\Gammao)$ be given. Let
$\delta>0$ be so small that the profile $P_0$ (see Fig.~1) is
on the left from the straight line $x_1=d-\delta$. Then there
exists $g_*\in W^{n+1,2}_{\rm per}(\Omega)$, such that
\vspace{-2pt}
\begin{displaymath}
g_*=g_0, \quad
\partial_1 g_*=g_1, \quad \dots, \quad
\partial_1^{n} g_*=g_n \qquad \mbox{on}\ \Gammao
\end{displaymath}

\vspace{-2pt} \noindent
in the sense of traces, $\supp g_*=\{\bfx=(x_1,x_2)\in\Omega;\
d-\delta\leq x_1\leq d\}$ and
\vspace{-2pt}
\begin{equation}
\|g_*\|_{n+1,2}\ \leq\ c\, \bigl( \|g_0\|_{n+1/2,2;\, \Gammao}+
\|g_1\|_{n-1/2,2;\, \Gammao}+\ldots+\|g_n\|_{1/2,2;\,
\Gammao}\bigr), \label{2.19}
\end{equation}

\vspace{-2pt} \noindent
where $c=c(\Omega,n,\delta)$.
\end{lemma}

\noindent
{\it Principles of the proof} \ The functions $g_0,\, g_1\,
\dots,\, g_n$ can be extended from $\Gammao$ to $\gammao$ so
that the extended functions (which we again denote by $g_0,\,
g_1\, \dots,\, g_n$) are in the spaces $W^{n+1/2}_{\rm
per}(\gammao)$, $W^{n-1/2,2}_{\rm per}(\gammao)$, $\dots,$
$W^{1/2,2}_{\rm per}(\gammao)$, respectively. Then we apply a
variant of Theorem II.4.4 in \cite{Ga}, which enables us to
deduce that there exists a function $\psi\in
W^{n+1,2}_{loc}(\R^2_{d-})$, $\tau$--periodic in variable
$x_2$, such that
\vspace{-2pt}
\begin{displaymath}
\psi=g_0, \quad \partial_1\psi=g_1, \quad \dots, \quad
\partial_1^{n}\psi=g_n \qquad \mbox{on}\ \Gammao
\end{displaymath}

\vspace{-2pt} \noindent
in the sense of traces and satisfying estimate (\ref{2.19}).
Note that Theorem II.4.4 from \cite{Ga} in fact deals with
functions $(g_0,\br g_1,\br \dots,\br g_n)\in
W^{n+1/2,2}(\gammao)\times\ W^{n-1/2,2}(\gammao)\times
\ldots\times W^{1/2,2}(\gammao)$ and the extension $\psi$ is in
the space $W^{n+1,2}(\R^2_{d-})$. However, the proof (which is
based on \cite[Chap.~2, Theorems 5.5, 5.8]{Ne}) can be modified
so that the theorem can also be applied to $(g_0,\br g_1,\br
\dots,\br g_n)\in W^{n+1/2,2}_{\rm per} (\gammao)\times\
W^{n-1/2,2}_{\rm per}(\gammao)\times \dots\times W^{1/2,2}_{\rm
per}(\gammao)$ and the extension is in
$W^{2,2}_{loc}(\R^2_{d-})$, $\tau$--periodic in variable $x_2$.
Then multiplying $\psi$ by an infinitely differentiable and
$\tau$--periodic in variable $x_2$ cut--off function $\eta$ in
$\R^2_{d-}$, such that $\eta=1$ in some neighborhood of
$\gammao$ and $\eta=0$ in the neighborhood of $\gammai$ and
$\Gammaw$, we obtain function $g_*$, whose restriction to
$\Omega$ has the properties stated in the lemma.

\begin{lemma} \label{L2.8}
There exists a bounded bilinear operator
$\cF:\bfL^2(\Omega)\times \bfW^{1/2,2}_{\rm per}(\Gammao)\to
W^{1,2}_{\rm per} (\Omega)^{2\times 2}$, such that if $\,
\bff\in\bfL^2(\Omega)$, $\, \bfh\in\bfW^{1/2,2}_{per}(\Gammao)$
and $\bbF=\cF(\bff,\bfh)$ then $\div\bbF=\bff$ a.e.~in
$\Omega$, $\bbF=\bbO$ (the zero tensor) on $\Gammaw$ and
$\bbF\cdot\bfn=\bfh$ a.e.~on $\Gammao$ in the sense of traces.
\end{lemma}

\begin{proof}
Denote by $\Omega_-$ the mirror image of $\Omega$ with respect
to the line $x_1=0$. Thus, $\Omega_-:=\{(x_1,x_2)\in\R^2;\
(-x_1,x_2)\in\Omega\}$. Furthermore, put $\Omegat:=\Omega_-\cup
\Gammai^0\cup\Omega$. We will construct $\bbF$ in the form
$\bbF=\bbF_0+\bbH_0+\bbH_1+\bbH_2$, where the tensor functions
$\bbF_0,\, \dots,\, \bbH_2$ are described below.

\vspace{4pt}
1) {\it Function $\bbF_0$:} \ Extend $\bff$ from $\Omega$ to
$\Omegat$ so that the extended function (we denote it also by
$\bff$) is odd in variable $x_1$. Then $\bff\in\bfL^2(\Omegat)$
and $\int_{\Omegat}\bff\; \rmd\bfx=\bfzero$. Due to
\cite[Theorem III.3.3] {Ga}, there exists $\bbF_0\in
W^{1,2}_0(\Omegat)^{2\times 2}$, such that $\div\bbF_0=\bff$ in
$\Omegat$ and $\|\bbF_0\|_{1,2;\, \Omegat}\leq c\,
\|\bff\|_{2;\, \Omegat}$, where $c=c(\Omegat)$. Hence we also
have
\begin{equation}
\|\bbF_0\|_{1,2}\ \leq\ c\, \|\bff\|_2. \label{2.20}
\end{equation}

\vspace{4pt}
2) {\it Function $\bbH_0$:} \ Put
$\overline{\bfh}\equiv(\overline{h}_1,\overline{h}_2):=
\tau^{-1} \int_{\Gammao}\bfh\, \rmd l$ and define
$\bfh_0\equiv(h_{01}, h_{02})^T:=\bfh-\overline{\bfh}$ on
$\Gammao$. Naturally, the function $\bfh_0$ is in
$\bfW^{1/2,2}_{\rm per}(\Gammao)$ and its advantage is that
$\int_{\Gammao}\bfh_0\, \rmd l=\bfzero$. We construct $\bbH_0$
so that its $i$--th row (for $i=1,2$) has the form
$\nabla^{\perp} \psi_i$, where $\nabla^{\perp}=
(\partial_2,-\partial_1)$ and $\psi_i$ is an appropriate
function from $W^{2,2}(\Omega)$, which is defined below, in two
steps:

2a) We define $\psi_i$ at first on the line segment $\Gammao$
by the formula
\begin{displaymath}
\psi_i(d,x_2):\ =\ \int_{b_{02}}^{x_2} h_{0i}(d,\vartheta)\,
\rmd\vartheta,
\end{displaymath}
where $b_{02}$ is the $x_2$--coordinate of point $B_0$. (See
Fig.~1.) Since $h_{0i}\in W^{1/2,2}_{per}(\Gammao)$ and
$\int_{b_{02}}^{b_{02}+\tau}h_{0i}(d,\vartheta)\,
\rmd\vartheta=0$, function $\psi_i$ is in $W^{3/2,2}_{\rm per}
(\Gammao)$. Obviously,
\begin{equation}
\partial_2\psi_i\ =\ h_{0i} \qquad \mbox{a.e.~on}\ \Gammao.
\label{2.41}
\end{equation}

2b) Applying Lemma \ref{L2.9}, we deduce that there exists an
extension of $\psi_i$ from $\Gammao$ to $\Omega$ (which we
again denote by $\psi_i$), such that $\psi_i\in W^{2,2}_{\rm
per}(\Omega)$, $\partial_1\psi_i=0$ on $\Gammao$,
\begin{equation} \|\psi_i\|_{2,2}\ \leq\ c\,
\|\psi_i\|_{3/2,2;\, \Gammao}\ \leq\ c\, \|h_{0i}\|_{1/2,2;\,
\Gammao} \label{2.43}
\end{equation}
and $\psi_i$ is supported in $\{\bfx=(x_1,x_2)\in\Omega;\
d-\delta\leq x_1\leq d\}$, where $\delta>0$ is so small that
all points $\bfx=(x_1,x_2)$ on $\Gammaw$ satisfy
$x_1<d-\delta$.

The scalar functions $\psi_i$ (for $i=1,2$) satisfy
(\ref{2.41}). The vector functions $\nabla^{\perp}\psi_i$ are
in $\bfW^{1,2}_{\rm per}(\Omega)$ and satisfy
$\nabla^{\perp}\psi_i\cdot\bfn=h_{0i}$ a.e.~on $\Gammao$. Since
the $i$--the row in the tensor function $\bbH_0$ equals
$\nabla^{\perp}\psi_i$, we have $\bbH_0\in W^{1,2}_{\rm
per}(\Omega)^{2\times 2}$ and $\div\bbH_0=\bfzero$. Moreover,
$\bbH_0=\bbO$ on $\Gammai\cup\Gammaw$, $\bbH_0\cdot\bfn=\bfh_0$
on $\Gammao$ and
\begin{equation}
\|\bbH_0\|_{1,2}\ \leq\ c\, \|\bfh_0\|_{1/2,2;\, \Gammao},
\label{2.25}
\end{equation}
where $c=c(\Omega)$.

\vspace{4pt}
3) {\it Function $\bbH_1$:} \ Denote by $\overline{\bbH}$ the
constant tensor with the entries
$\overline{H}_{11}=\overline{h}_1$, $\overline{H}_{12}=0$,
$\overline{H}_{21}=\overline{h}_2$, $\overline{H}_{22}=0$. Let
$\zeta$ be an even infinitely differentiable function of one
variable $x_1$ for $x_1\in[-d,d]$, such that
$\zeta(-d)=\zeta(d)=1$ and
$\supp\zeta\subset[-d,-d+\delta]\cup[d-\delta,d]$, where
$\delta>0$ is so small that the profile $P_0$ lies on the left
from the straight line $x_1=d-\delta$. Define
$\bbH_1(x_1,x_2):=\zeta(x_1)\, \overline{\bbH}$. The tensor
function $\bbH_1$ is in $W^{1,2}_{\rm per}(\Omega)^{2\times 2}$
and satisfies $\bbH_1\cdot\bfn=\overline{\bfh}$ on $\Gammao$,
$\bbH_1=\bbO$ on $\Gammai\cup\Gammaw$ and
$\div\bbH_1=\zeta'(x_1)\, \overline{\bfh}$ in $\Omegat$.

\vspace{4pt}
4) {\it Function $\bbH_2$:} \ Since
$\int_{\Omegat}\zeta'(x_1)\, \overline{\bfh}\;
\rmd\bfx=\bfzero$, there exists (by \cite[Theorem III.3.3]{Ga})
a tensor function $\bbH_2\in W^{1,2}_0(\Omegat)^{2\times 2}$,
satisfying $\div\bbH_2=-\zeta'(x_1)\, \overline{\bfh}$ in
$\Omegat$. The restriction of $\bbH_2$ to $\Omega$ is in
$W^{1,2}_{\rm per}(\Omega)^{2\times 2}$ and $\bbH_2$
particularly satisfies $\bbH_2\cdot\bfn=\bfzero$ on $\Gammao$.

\vspace{4pt}
Using the properties of $\bbF_0$, $\bbH_0$, $\bbH_1$ and
$\bbH_2$, we observe that $\bbF:=\bbF_0+\bbH_0+\bbH_1+\bbH_2$
has all the properties stated in the lemma. The whole procedure
can be formalized so that the mapping $(\bff,\bfh)\mapsto\bbF$
is a bilinear operator. One can simply derive from
(\ref{2.20}), (\ref{2.25}) and the definition of $\bbH_1$ and
$\bbH_2$ that this operator is bounded from
$\bfL^2(\Omega)\times \bfW^{1/2,2}_{\rm per}(\Omega)$ to
$W^{1,2}_{\rm per} (\Omega)^{2\times 2}$.
\end{proof}

\vspace{4pt}
The next lemma generalizes Lemma \ref{L2.5}:

\begin{lemma} \label{L2.1}
Let $m\in\{0\}\cup\N$ and $\bfg\in\bfW^{m+1/2,2}_{\rm
per}(\Gammai)$ be given. Then there exists a divergence--free
extension $\bfgs\in\bfW^{m+1,2}_{\rm per}(\Omega)$ with the
properties c) and d) from Lemma \ref{L2.5}, such that
\begin{equation}
\|\bfg_*\|_{m+1,2}\ \leq\ c\, \|\bfg\|_{m+1/2,2;\, \Gammai},
\label{2.21}
\end{equation}
where $c=c(\Omega,m)$.
\end{lemma}

\noindent
{\it Principles of the proof} \ Let $g_1$, $g_2$ be the
components of $\bfg$. The function $g_1$ can be written in the
form $g_1=g_1^0+\overline{g_1}$, where
$\overline{g_1}:=\tau^{-1}\int_{\Gammai}g_1\; \rmd l$. Then
$\int_{\Gammai}g_1^0\; \rmd l=0$.

Put $\psi(0,x_2):= \int_{a_{02}}^{x_2} g_1^0(0,\vartheta)\,
\rmd\vartheta$ for $x_2\in(a_{02},a_{02}+\tau)$, where $a_{02}$
is the $x_2$--component of point $A_0$. (See Fig.~1.) Then
$\psi\in W^{m+3/2,2}_{\rm per}(\Gammai)$. We may apply Lemma
\ref{L2.9} (with $\Gammai$ instead of $\Gammao$ and $n=m$) to
extend $\psi$ from $\Gammai$ to $\Omega$ so that the extended
function $\psi_*$ is in $W^{m+2,2}_{\rm per}(\Omega)$ and
\begin{displaymath}
\partial_1\psi = -g_2 \quad \mbox{and} \quad
\partial_1^k \psi\ =\ 0 \quad (k=2,\dots,m+1)
\end{displaymath}
on $\gammai$ in the sense of traces and
\begin{align*}
\|\psi\|_{m+2,2}\ &\leq\ c\, \bigl( \|\psi\|_{m+3/2,2;\, \Gammai}+
\|\partial_1\psi\|_{m+1/2,2;\, \Gammai}+\ldots+
\|\partial_1^{m+1}\psi\|_{1/2,2;\, \Gammai} \bigr) \\
\noalign{\vskip 2pt}
&\leq\ c\, \bigl( \|\partial_2\psi\|_{m+1/2,2;\,
\Gammai}+\|\partial_1\psi\|_{m+1/2,2;\, \Gammai}\bigr) \\
\noalign{\vskip 2pt}
&=\ c\, \bigl( \|g_1^0\|_{m+1/2,2;\, \Gammai}+\|g_2\|_{m+1/2;\,
\Gammai} \bigr).
\end{align*}
Define $\bfg_*^0:=\nabla^{\perp}\psi$. Then $\bfg_*^0$ is a
divergence--free function in $\Omega$, belongs to
$\bfW^{m+1,2}_{\rm per}(\Omega)$, satisfies the inequality
$\|\bfg_*^0\|_{m+1,2}\leq c\, \|(g_1^0,g_2)\|_{m+1/2,2;\,
\Gammai}$ and its trace on $\Gammai$ equals $(g_1^0,g_2)$. Put
$\bfg_*:=\bfg_*^0+(\overline{g_1},0)^T$.

Further steps, which modify function $\bfg_*$ so that it also
has the properties c) and d) from Lemma \ref{L2.5}, can be made
in the same way as in \cite[Sec.~3]{FeNe3}.

\begin{theorem} \label{T2}
[on a strong solution of the Stokes problem
(\ref{1.2})--(\ref{1.8}), (\ref{1.11})] Let the curve $\Gammaw$
(which is the boundary of the profile) be of the class $C^2$,
$\bbF\in W^{1,2}_{\rm per}(\Omega)^{2\times 2}$ and $\bfg$,
$\bfg_*$ be the functions from Lemma \ref{L2.1}, corresponding
to $m=1$. Let the functionals $\bfF$ and $\bfG$ be defined by
formulas (\ref{2.1}) and (\ref{2.34}), respectively. Then

\begin{list}{}
{\setlength{\topsep 0pt}
\setlength{\itemsep 1pt}
\setlength{\leftmargin 19pt}
\setlength{\rightmargin 0pt}
\setlength{\labelwidth 14pt}}

\item[(a)]
the unique solution $\bfv$ of the equation
$\nu\cA\bfv=\bfF+\nu\br\bfG$ belongs to $\Vs\cap\bfW^{2,2}_{\rm
per}(\Omega)$ and the associated pressure $p$ is in
$W^{1,2}_{\rm per} (\Omega)$,

\item[(b)]
the functions $\bfu:=\bfg_*+\bfv$ and $p$ satisfy equations (\ref{1.2})
and (\ref{2.33}) a.e.~in $\Omega$,

\item[(c)]
$\bfu$, $p$ satisfy the boundary conditions (\ref{1.3}),
(\ref{1.4}) and
\begin{equation}
-\nu\, \frac{\partial\bfu}{\partial\bfn}+p\br\bfn\ =\
\bbF\cdot\bfn \label{2.18}
\end{equation}
in the sense of traces on $\Gammai$, $\Gammaw$ and $\Gammao$,
respectively,

\item[(d)]
there exists a constant $\cn01=\cc01(\nu,\Omega)$, such that
\begin{equation}
\|\bfu\|_{2,2}+\|\nabla p\|_2\ \leq\ \cc01\, \bigl(
\|\bbF\|_{1,2}+\|\bfg_*\|_{2,2} \bigr). \label{2.7}
\end{equation}
\end{list}
\end{theorem}

Note that the existence and uniqueness of the solution
$\bfv\in\Vs$ of the equation $\nu\cA\bfv=\bfF+\nu\br\bfG$
follows from Lemma \ref{L2.5} or Theorem \ref{T1}.

If $\bff\in\bfL^2(\Omega)$ and
$\bfh\in\bfW^{1/2,2}_{per}(\Omega)$ are given and $\bbF$ is the
tensor function, provided by Lemma \ref{L2.2}, then functions
$\bfu$ and $p$ from Theorem \ref{T2} represent a strong solution to
the Stokes problem (\ref{1.11}), (\ref{1.2})--(\ref{1.8}),
where $\bff=\div\bbF$ in equation (\ref{1.11}).

The conclusions $\bfu\in\bfW^{2,2}(\Omega)$ and $p\in
W^{1,2}(\Omega)$ of Theorem \ref{T2}, together with inequality
(\ref{2.7}), represent the maximum regularity property of the
studied problem.

\vspace{10pt} \noindent
{\it Proof of Theorem \ref{T2}} \ Put
$\bfh\equiv(h_1,h_2)^T:=\bbF\cdot\bfn$ on $\Gammao$. The tensor
function $\bbF$ can be written in the form
$\bbF=\bbF_1+\bbF_2$, where both $\bbF_1$ and $\bbF_2$ are in
$W^{1,2}_{\rm per}(\Omega)^{2\times 2}$,
$\bbF_1\cdot\bfn=h_1\br\bfe_1$ on $\Gammao$ and
$\bbF\cdot\bfn=h_2\br\bfe_2$ on $\Gammao$. Denote by $\bfF_1$
and $\bfF_2$ the functionals in $\Vsd$, related to $\bbF_1$ and
$\bbF_2$, respectively, through formula (\ref{2.1}).

We split the proof of Theorem \ref{T2} to seven parts, where we
successively prove

\begin{list}{}
{\setlength{\topsep 1pt}
\setlength{\itemsep 2pt}
\setlength{\leftmargin 21pt}
\setlength{\rightmargin 0pt}
\setlength{\labelwidth 12pt}}

\item[1) ]
the implication (a) $\Longrightarrow$ (b), (c),

\item[2) ]
the implication (a), (b), (c) $\Longrightarrow$ (d),

\item[3) ]
the solvability of the equation $\nu\cA\bfv=\bff$ in
$\Vs\cap\bfW^{2,2}(\Omega)$ for $\bff\in\Ls$,

\item[4) ]
the solvability of the equation $\nu\cA\bfv_1=\bfF_1$ in
$\Vs\cap\bfW^{2,2}(\Omega)$ and the inclusion \\ $p_1\in
W^{1,2}(\Omega)$ for an associated pressure,

\item[5) ]
the solvability of the equation $\nu\cA\bfv_2=\bfF_2$ in
$\Vs\cap\bfW^{2,2}(\Omega)$ and the inclusion \\ $p_2\in
W^{1,2}(\Omega)$ for an associated pressure,

\item[6) ]
the solvability of the equation $\nu\cA\bfv_3=\nu\bfG$ in
$\Vs\cap\bfW^{2,2}(\Omega)$ and the inclusion \\ $p_3\in
W^{1,2}(\Omega)$ for an associated pressure,

\item[7) ]
the validity of statement (a).

\end{list}

\noindent
Note that this scheme does not create a logical circle, because we do not use the validity of statement (a) of Theorem \ref{T2} in parts 2) -- 6). We use the validity of the implication (a) $\Longrightarrow$ (b), (c), which is not the same as the validity of (a).

The most technical part of the proof is the derivation of inequality (\ref{2.7}) in part 2). This inequality is further used in part 3) in order to show that a certain Stokes--type operator $A$ is closed and its range is closed in $\Ls$.

\vspace{-4pt}
\paragraph{1) \ The implication (a) $\Longrightarrow$ (b), (c).}
As equation (\ref{2.33}) is satisfied in the sense of
distributions in $\Omega$ (due to Theorem \ref{T1}) and all terms in
this equation are now in $\bfL^2(\Omega)$, the equation is
satisfied a.e.~in $\Omega$. Clearly, $\bfu$ also satisfies
equation (\ref{1.2}) a.e.~in $\Omega$ and boundary conditions
(\ref{1.3}) and (\ref{1.4}) in the sense of traces on $\Gammai$
and $\Gammaw$, respectively. Since $\bfu$ and $p$ satisfy the
boundary condition (\ref{2.31}) in the sense of equality in
$\bfW^{-1/2,2}(\Gammao)$ (see Theorem \ref{T1}) and all the functions
$\nabla\bfu$, $p\br\bbI$ and $\bbF$ have traces on $\Gammao$ in
$W^{1/2,2}(\Gammao)^{2\times 2}$, the boundary condition
(\ref{2.31}) holds on $\Gammao$ in the sense of traces, too. It
can now be written in the form (\ref{2.18}).

\vspace{-6pt}
\paragraph{2) \ The implication (a), (b), (c) $\Longrightarrow$
(d).} \ We split the proof of the implication to three lemmas,
where we successively derive an inequality, analogous to
(\ref{2.7}), in the interior of $\Omega$ plus the neighborhood
of $\Gammaw$ and $\Gammai^0$ (Lemma \ref{L2.2}), in the
neighborhood of $\Gammao^0$ (Lemma \ref{L2.6}) and in the
neighborhoods of $\Gammam$ and $\Gammap$ (Lemma \ref{L2.7}).

\begin{lemma} \label{L2.2}
Let $\Omega'$ be sub-domain of $\Omega$, such that
$\overline{\Omega'}\subset\Omega\cup\Gammai^0\cup\Gammaw$. Then
\begin{equation}
\|\bfv\|_{2,2;\, \Omega'}+\|\nabla p\|_{2;\, \Omega'}\ \leq\ c\,
\bigl( \|\div\bbF\|_{2}+\|\bfg_*\|_{2,2}+\|\bfv\|_{1,2} \bigr),
\label{2.17}
\end{equation}
where $c=c(\nu,\Omega,\Omega')$.
\end{lemma}

\begin{proof}
Consider a $C^2$ sub-domain $\Omega''$ of $\Omega$, such that
$\Omega'\subset\Omega''$,
$\overline{\Omega''}\subset\Omega\cup\Gammai^0\cup\Gammaw$ and
the distance between $\partial\Omega''\cap\Omega$ and
$\partial\Omega'\cap\Omega$ is positive. Let $\eta$ be an
infinitely differentiable cut--off function in $\Omega$ such
that $\supp\eta\subset\overline{\Omega''}$ and $\eta=1$ in
$\Omega'$. Put $\bfvt:=\eta\bfv$ and $\pt:=\eta\br p$. Since
$\bfv$, $p$ satisfy (\ref{2.33}) a.e.~in $\Omega$, the
functions $\bfvt$, $\pt$ represent a strong solution of the
problem
\begin{align}
-\nu\Delta\bfvt+\nabla\pt\ &=\ \bfft && \mbox{in}\ \Omega'',
\label{2.12} \\
\div\bfvt\ &=\ \hht && \mbox{in}\ \Omega'', \label{2.13} \\
\bfvt\ &=\ \bfzero && \mbox{on}\ \partial\Omega'', \label{2.14}
\end{align}
where
\begin{displaymath}
\bfft:=\eta\, \div\bbF-2\nu\br\nabla\eta\cdot\nabla\bfv-\nu\,
(\Delta\eta)\, \bfv-(\nabla\eta)\, p+\nu\br\eta\br\Delta\bfg_*
\qquad \mbox{and} \qquad \hht:=\nabla\eta\cdot\bfv.
\end{displaymath}
As $\div\bbF\in\bfL^2(\Omega)$, $\bfv\in\Vs$ and $p\in
L^2(\Omega)$ (satisfying (\ref{2.29})), we have
$\bfft\in\bfL^2(\Omega)$ and $\hht\in W^{1,2}(\Omega)$.
Moreover,
\begin{alignat}{3}
& \|\bfft\|_2\ &&\leq\ c\, \bigl( \|\div\bbF\|_2+
\|\bfg_*\|_{2,2}+\|\bfv\|_{1,2} \bigr), \label{2.15} \\
& \|\hht\|_{1,2}\ &&\leq\ c\, \|\bfv\|_2\ \leq\ c\,
\|\bfv\|_{1,2}, \label{2.16}
\end{alignat}
where $c=c(\nu,\eta)$. Due to \cite[Proposition I.2.3]{Te},
\begin{displaymath}
\|\bfvt\|_{2,2;\, \Omega''}+\|\nabla\pt\|_{2;\, \Omega''}\
\leq\ c\, \bigl( \|\bfft\|_{2;\, \Omega''}+\|\hht\|_{1,2;\,
\Omega''} \bigr),
\end{displaymath}
where $c=c(\nu,\Omega'')$. This inequality, together with
(\ref{2.15}) and (\ref{2.16}), implies that $\bfv$ and $p$
satisfy (\ref{2.17}).
\end{proof}

\begin{lemma} \label{L2.6}
Let $\Gammao'$ be a (closed) line segment on $\Gammao^0$. For
$\rho>0$, denote $\Omega':=\{ \bfx\in\R^2_{d-};\
\dist(\bfx;\Gammao')<\rho\}$. Assume that $\rho$ is so small
that $\overline{\Omega'}$ is disjoint with $\Gammam$, $\Gammap$
and $\Gammaw$. Then $\bfv$ and $p$ satisfy estimate
(\ref{2.17}), where the constant $c$ on the right hand side
again depends only on $\nu$, $\Omega$ and $\Omega'$.
\end{lemma}

\noindent
\begin{minipage}{70mm}
{\it Proof} \
The condition of smallness of $\rho$ guarantees that
$\overline{\Omega'}\subset\Omega\cup\Gammao^0$.

\hspace{10pt}
Denote $U':=\{\bfx=$ $(x_1,x_2)\in\R^2;\
\dist(\bfx$, $\Gammao')<\rho\}$. ($U'$ is the $\rho$--neighborhood
of $\Gammao'$ in $\R^2$.) Let $\delta>0$. By analogy with $U'$,
denote by $U''$ the $(\rho+\delta)$--neighborhood of $\Gammao'$
in $\R^2$. Denote further by $U''_{d-}$ the intersection of
$U''$ with the half-plane $\R^2_{d-}$ and by $U''_{d+}$ the
intersection of $U''$ with the half-plane
$\R^2_{d+}:=\{\bfx=(x_1,x_2)\in\R^2;\ x_1>d\}$. Assume that
$\delta$ is so small that the distances between $U''$ and
points $B_0$, $B_1$ are positive and $U''_{d-}\subset\Omega$.
(See Fig.~2.)

\hspace{10pt}
Let $\eta$ be a $C^{\infty}$--function in $\R^2$, supported in
$\overline{U''}$, such that $\eta=1$ in $U'$ and $\eta$ is
symmetric with respect to the line $x_1=d$. (The last condition
means that $\eta(d+\vartheta,x_2)=\eta(d-\vartheta,x_2)$ for
all $\vartheta,\, x_2\in\R$.)
\end{minipage}
\begin{minipage}{60mm}
  \setlength{\unitlength}{0.35mm}
  \begin{picture}(162,155)
  \put(18,90){\vector(1,0){123}}
  \put(128,94){\small $x_1$}
  %
  \put(77.3,50){\line(0,1){90}}
  \qbezier(77,50)(47,64)(30,66) \qbezier(77,140)(47,154)(30,157)
  \dashline[+30]{2.2}(77.3,32)(77.3,164)
  \thicklines
  \put(77.3,75){\line(0,1){35}}
  \thinlines
  \qbezier[16](78,100)(89,102)(100,104) \put(102,102){\small $\Gammao'$}
  \qbezier(76,75)(77.3,75)(78.2,75)
  \qbezier(76,110)(77.3,110)(78.2,110)
  %
  \qbezier(77.3,124)(83.096,124)(87.198,119.898) 
  \qbezier(87.198,119.898)(91.3,115.796)(91.3,110) 
  \qbezier(91.3,75)(91.3,69.204)(87.198,65.102) 
  \qbezier(87.198,65.102)(83.096,61)(77.3,61) 
  \qbezier(77.3,124)(71.504,124)(67.402,119.898) 
  \qbezier(67.402,119.898)(63.3,115.796)(63.3,110) 
  \qbezier(63.3,75)(63.3,69.204)(67.402,65.102) 
  \qbezier(67.402,65.102)(71.504,61)(77.3,61) 
  \put(63.4,110){\line(0,-1){35}} \put(91.4,110){\line(0,-1){35}}
  \qbezier(77.3,118)(80.612,118)(82.956,115.656)
  \qbezier(82.956,115.656)(85.3,113.312)(85.3,110)
  \qbezier(85.3,75)(85.3,71.688)(82.956,69.344)
  \qbezier(82.956,69.344)(80.612,67)(77.3,67)
  \qbezier(77.3,118)(73.988,118)(71.644,115.656)
  \qbezier(71.644,115.656)(69.3,113.312)(69.3,110)
  \qbezier(69.3,75)(69.3,71.688)(71.644,69.344)
  \qbezier(71.644,69.344)(73.988,68)(77.3,67)
  \put(69.4,110){\line(0,-1){35}} \put(85.4,110){\line(0,-1){35}}
  \qbezier[11](85.4,80)(93.1,80)(100,80)
  \put(102,78){\small $U'$}
  \put(60,131){\small $\Gammao$}
  \put(40,143){\small $\Gammap$} \put(40,53){\small $\Gammam$}
  \put(44,111){\small $\Omega$}
  \put(80,46){\small $B_0$} \put(80,137){\small $B_1$}
  \put(81,157){\small $\gammao$ ($x_1=d$)}
  \put(92,62){\small $U''$}
  \put(18,12){\small Fig.~2: \ The line segment $\Gammao'$ and its}
  \put(43.5,1){\small neighborhoods $U'$ and $U''$}
  \end{picture}
\end{minipage}

Applying the results from \cite{KMPT}, one can deduce that
there exists a divergence--free extension $\bfv''$ of function
$\bfv$ from $U''_{d-}$ to the whole set $U''$, such that
$\bfv''\in\bfW^{1,2}(U'')$ and $\| \bfv'' \|_{1,2;\, U''}\leq
c\, \|\bfv\|_{1,2;\, U''_{d-}}$, where $c$ is independent of
$\bfv$.

Since $\nabla\eta\cdot\bfv''\in W^{1,2}_0(U'')$ and
$\int_{U''}\nabla\eta\cdot\bfv''\; \rmd\bfx=0$, there exists
(by \cite[Theorem III.3.3]{Ga}) $\bfv_*\in\bfW^{2,2}_0(U'')$,
such that $\div\bfv_*=\nabla\eta\cdot\bfv''$ in $U''$ and
\begin{displaymath}
\|\bfv_*\|_{2,2;\, U''}\ \leq\ c\,
\|\nabla\eta\cdot\bfv''\|_{1,2;\, U''}\ \leq\ c\,
\|\bfv''\|_{1,2;\, U''}\ \leq\ c\, \|\bfv\|_{1,2;\, U''_{d-}},
\end{displaymath}
where $c$ is independent of $\bfv$. Extending $\bfv_*$ by zero
to $\Omega\smallsetminus U''$, we have $\|\bfv_*\|_{2,2}\leq
c\, \|\bfv\|_{1,2}$. Put
\begin{equation}
\bfvt\ :=\ \eta\bfv''-\bfv_*, \qquad \pt\ :=\ \eta\br p.
\label{2.32}
\end{equation}
Function $\bfvt$ is divergence--free, belongs to
$\bfW^{1,2}_0(U'')$ and satisfies the estimates
\begin{displaymath}
\|\bfvt\|_{1,2;\, U''}\ \leq\ c\, \bigl( \|\bfv''\|_{1,2;\,
U''}+\|\bfv_*\|_{1,2;\, U''} \bigr)\ \leq\ c\,
\|\bfv''\|_{1,2;\, U''}\ \leq\ c\, \|\bfv\|_{1,2},
\end{displaymath}
where $c$ is independent of $\bfv$. The functions $\bfvt$,
$\pt$ satisfy equation (\ref{2.12}) a.e.~in the half-plane
$\R^2_{d-}$, where function $\bfft$ now satisfies
\begin{displaymath}
\bfft\ :=\ \eta\,
\div\bbF-2\nu\br\nabla\eta\cdot\nabla\bfv-\nu\, (\Delta\eta)\,
\bfv-(\nabla\eta)\, p+\nu\br\eta\Delta \bfg_*+\nu\Delta\bfv_*
\qquad \mbox{in}\ U''_{d-}
\end{displaymath}
and $\bfft:=\bfzero$ in $\R^2_{d-}\smallsetminus U''_{d-}$.
Although this function differs from the function $\bfft$ from
the proof of Lemma \ref{L2.2}, it satisfies the same estimate
(\ref{2.15}). Define
\begin{displaymath}
\bfht\ :=\ \eta\br\bbF\cdot\bfn+\nu\,
\frac{\partial\bfv_*}{\partial\bfn} \qquad \mbox{on}\ \Gammao,
\end{displaymath}
where the right hand side is understood as a trace on
$\Gammao$. The function $\bfht$ satisfies
\begin{align}
\|\bfht\|_{1/2,2;\, \Gammao}\ &\leq\ c\, \|\bbF\|_{1/2,2;\,
\Gammao}+\Bigl\| \frac{\partial\bfv_*}{\partial\bfn}
\Bigr\|_{1/2,2;\, \Gammao}\ \leq\ c\, \|\bbF\|_{1,2}+c\,
\|\bfv_*\|_{2,2;\, U''} \nonumber \\
&\leq\ c\, \|\bbF\|_{1,2}+c\, \|\bfv\|_{1,2;\, U''_{d-}}.
\label{2.22}
\end{align}
Put $\bbFt:=\cF(\bfft,\bfht)$, where $\cF$ is the operator from
Lemma \ref{L2.8}. Then $\bbFt\in W^{1,2}(\Omega)^{2\times 2}$,
$\bfft=\div\bbFt$ a.e.~in $\Omega$ and $\bbFt\cdot\bfn=\bfht$
a.e.~on $\Gammao$. Moreover, due to (\ref{2.15}) and
(\ref{2.22}),
\begin{equation}
\|\bbFt\|_{1,2}\ \leq\ c\, \|\bfft\|_2+c\, \|\bfht\|_{1/2,2;\,
\Gammao}\ \leq\ c \bigl( \|\bbF\|_{1,2}+\|\bfgs\|_{2,2}+
\|\bfv\|_{1,2} \bigr). \label{2.23}
\end{equation}

Let the functional $\bfFt\in\Vsd$ be defined by the same
formula as (\ref{2.1}), where we only consider $\bbFt$ instead
of $\bbF$. We claim that $\nu\cA\bfvt=\bfFt$. Indeed, for any
$\bfw\in\Vs$, we have
\begin{align*}
\nu\, \langle\cA & \bfvt,\bfw\rangle_{\sigma}\ =\ \nu\int_{\Omega}
\nabla\bfvt:\nabla\bfw\; \rmd\bfx\ =\ \int_{\Gammao}\nu\,
\frac{\partial\bfvt}{\partial\bfn}\cdot\bfw\; \rmd
l-\nu\int_{\Omega}\Delta\bfvt\cdot\bfw\; \rmd\bfx \\
&= \int_{\Gammao}\nu\, \frac{\partial\bfvt}{\partial\bfn}\cdot
\bfw\; \rmd l+\int_{\Omega}(-\nabla\pt+\bfft)\cdot\bfw\; \rmd\bfx \\
&= \int_{\Gammao} \Bigl[ \nu\,
\frac{\partial\bfvt}{\partial\bfn}-\pt\br\bfn \Bigr]\cdot\bfw\;
\rmd l+\int_{\Omega}\div\bbFt\cdot\bfw\; \rmd\bfx \\
&= \int_{\Gammao} \Bigl[ \eta\, \Bigl(\nu\,
\frac{\partial\bfv}{\partial\bfn}-p\br\bfn\Bigr)-\nu\,
\frac{\partial\bfv_*}{\partial\bfn} \Bigr]\cdot\bfw\; \rmd l+
\int_{\Gammao}(\bbFt\cdot\bfn)\cdot\bfw\; \rmd l-\int_{\Omega}
\bbFt:\nabla\bfw\; \rmd\bfx \\
&= \int_{\Gammao} \Bigl[ -\eta\br\bbF\cdot\bfn-\nu\,
\frac{\partial\bfv_*}{\partial\bfn} \Bigr]\cdot\bfw\; \rmd l+
\int_{\Gammao}(\bbFt\cdot\bfn)\cdot\bfw\; \rmd l-\int_{\Omega}
\bbFt:\nabla\bfw\; \rmd\bfx \\
&= -\int_{\Gammao} \bfht\cdot\bfw\; \rmd l+
\int_{\Gammao}(\bbFt\cdot\bfn)\cdot\bfw\; \rmd l-\int_{\Omega}
\bbFt:\nabla\bfw\; \rmd\bfx \\
&= -\int_{\Omega} \bbFt:\nabla\bfw\; \rmd\bfx\ =\
\langle\bfFt,\bfw\rangle_{\sigma}.
\end{align*}
(We have used the identity $\nu\,
\partial\bfv/\partial\bfn-p\br\bfn=-\bbF\cdot\bfn$ on $\Gammao$,
following from (\ref{2.18}) and the fact that
$\partial\bfgs/\partial\bfn=\bfzero$ on $\Gammao$.)

Let us summarize that we have constructed functions $\bfvt$,
$\pt$ and $\bbFt$, such that $\bfvt$ satisfies the equation
$\nu\cA\bfvt=\bfFt$ and $\pt$ is an associated pressure. The
functions $\bfvt$ and $\pt$ are supported in
$\overline{U''_{d-}}$ and $\bfvt$, $\pt$ are related to $\bfv$,
$p$ through formulas (\ref{2.32}).

Recall that $\bfft$ is supported in $\overline{U''_{d-}}$ and
$\bfht$ is supported in $\overline{U''_{d-}}\cap\Gammao$. For
$\delta\in\R$, whose modulus is so small that
$(x_1,x_2+\delta)\in\Omega$ for all $\bfx=(x_1,x_2)\in
U''_{d-}$, denote
\vspace{-4pt}
\begin{displaymath}
D_2^{\delta}\bfft(x_1,x_2):=\frac{\bfft(x_1,x_2+\delta)-
\bfft(x_1,x_2)}{\delta}, \quad
D_2^{\delta}\bfht(d,x_2):=\frac{\bfht(d,x_2+\delta)-
\bfht(d,x_2)}{\delta}.
\end{displaymath}
$D_2^{\delta}\bfft$ and $D_2^{\delta}\bfht$ are the so called
{\it difference quotients,} see \cite{Ag}, \cite{Gr1} and
\cite{So} for more details regarding their properties and usage
in studies of regularity of solutions of PDE's.

As $\bbFt\in W^{1,2}_{\rm per}(\Omega)^{2\times 2}$ and $\bbFt=
\bbO$ on $\Gammaw$, it can be extended from $\Omega$ to
$\R^2_{(0,d)}$ as a $\tau$--periodic function in variable
$x_2$, lying in $W^{1,2}_{loc}(\R^2_{(0,d)})$ and being equal
to $\bbO$ in $P_k$ (for all $k\in\Z$). Let us denote the
extension again by $\bbFt$ and define
\begin{displaymath}
D_2^{\delta}\bbFt(x_1,x_2)\ :=\ \frac{\bbFt(x_1,x_2+\delta)-
\bbFt(x_1,x_2)}{\delta}.
\end{displaymath}
Denote $\bbFt_{\delta}(x_1,x_2):=\delta^{-1}
\int_0^{\delta}\bbFt(x_1,x_2+ \vartheta)\; \rmd\vartheta$. Then
\begin{displaymath}
D_2^{\delta}\bbFt(x_1,x_2)\ =\
\frac{1}{\delta}\int_0^{\delta}\partial_2\bbFt(x_1,x_2+\vartheta)\;
\rmd\vartheta\ =\ \partial_2\bbFt_{\delta}(x_1,x_2).
\end{displaymath}
Furthermore, using the $\tau$--periodicity of the function
$\bbFt_{\delta}$ in variable $x_2$ in $\R^2_{(0,d)}$, we get
\begin{align*}
\|\bbFt_{\delta} & \|_2^2\ =\ \int_{\Omega}\biggl|\frac{1}{\delta}
\int_0^{\delta} \bbFt(x_1,x_2+\vartheta)\; \rmd
\vartheta\biggr|^2\; \rmd\bfx \\
&=\ \int_0^d\int_0^{\tau}
\biggl|\frac{1}{\delta} \int_0^{\delta} \bbFt(x_1,x_2+\vartheta)\;
\rmd \vartheta\biggr|^2\; \rmd x_2\, \rmd x_1 \nonumber \\
&\leq\ \int_0^d\int_0^{\tau}\frac{1}{\delta}
\int_0^{\delta}\bigl|\bbFt(x_1,x_2+\vartheta)\bigr|^2\;
\rmd\vartheta\, \rmd x_2\, \rmd x_1 \\
&=\ \int_0^d\frac{1}{\delta}
\int_0^{\delta}\int_0^{\tau}\bigl|\bbFt(x_1,y_2)\bigr|^2\;
\rmd y_2\, \rmd\vartheta\, \rmd x_1 \nonumber \\
&=\ \int_0^d \int_0^{\tau}\bigl|\bbFt(x_1,y_2)\bigr|^2\; \rmd
y_2\, \rmd x_1\ =\ \int_{\Omega}\bigl|\bbFt(\bfx)\bigr|^2\;
\rmd\bfx\ =\ \|\bbFt\|_2^2.
\end{align*}
We can similarly show that $\|\nabla\bbFt_{\delta}\|_2^2\leq
\|\nabla\bbFt\|_2^2$. Consequently,
$\|\bbFt_{\delta}\|_{1,2}\leq \|\bbFt\|_{1,2}$. Thus,
\begin{equation}
\| D_2^{\delta}\bbFt\|_2\ =\ \|\partial_2\bbFt_{\delta}\|_2\ \leq\
\|\bbFt_{\delta}\|_{1,2}\ \leq\ \|\bbFt\|_{1,2}\ \leq\ c\, \bigl(
\|\bfft\|_2+\|\bfht\|_{1/2,2;\, \Gammao}\bigr). \label{2.36}
\end{equation}
Let $D_2^{\delta}\bfvt$ and $D_2^{\delta}\pt$ be defined by
analogy with $D_2^{\delta}\bfft$ and $D_2^{\delta}\bbFt$. The
functions $D_2^{\delta}\bfvt$, $D_2^{\delta}\pt$ satisfy the
equations
\begin{align*}
-\nu\Delta D_2^{\delta}\bfvt+\nabla D_2^{\delta}\pt\ &=\
\div D_2^{\delta}\bbFt, \\
\div D_2^{\delta}\bfvt\ &=\ 0
\end{align*}
a.e.~in $\Omega$. Since
\begin{displaymath}
\nu\, \frac{\partial\bfvt}{\partial\bfn}-\pt\br\bfn\ =\ \nu\,
\bigl( \eta\, \frac{\partial\bfv}{\partial\bfn}-\frac{\partial
\bfv_*}{\partial\bfn}\Bigr)-\eta\br p\br\bfn\ =\ -\eta\,
\bbF\cdot\bfn-\nu\, \frac{\partial\bfv_*}{\partial\bfn}\ =\
-\bfht
\end{displaymath}
on $\gammao$, $D_2^{\delta}\bfvt$ and $D_2^{\delta}\pt$ also
satisfy the boundary condition
\begin{displaymath}
-\nu\, \frac{\partial
D_2^{\delta}\bfvt}{\partial\bfn}+D_2^{\delta}\pt\br\bfn\ =\
D_2^{\delta}\bfht
\end{displaymath}
on $\Gammao$. From this, one can deduce that $\nu\cA
D_2^{\delta}\bfvt=\bfFt_{\delta}$. Here, the functional
$\bfFt_{\delta}$, which is an element of $\Vsd$, is defined by
the same formula as (\ref{2.1}), where we only consider
$D_2^{\delta}\bbFt$ instead of $\bbF$. It follows from Lemma
\ref{L2.4} that
\begin{displaymath}
\|\nabla D_2^{\delta}\bfvt\|_2\ \leq\
\|\bfFt_{\delta}\|_{\vsd}.
\end{displaymath}
Since $\|\bfFt_{\delta}\|_{\vsd}\leq\|
D_2^{\delta}\bbFt\|_2\leq c\, \bigl(
\|\bfft\|_2+\|\bfht\|_{1/2,2;\, \Gammao}\bigr)$, we obtain
\begin{equation}
\|\nabla D_2^{\delta}\bfvt\|_2\ \leq\ c\, \bigl(
\|\bfft\|_2+\|\bfht\|_{1/2,2;\, \Gammao}\bigr). \label{2.37}
\end{equation}
Applying further Theorem \ref{T1} (with $\bfg_*=\bfzero$),
(\ref{2.36}) and (\ref{2.37}), we obtain the estimate of
$D_2^{\delta}\pt$:
\begin{equation}
\|D_2^{\delta}\pt\|_2\ \leq\ \cc03\, \bigl( \|\nabla
D_2^{\delta}\bfvt\|_2+\| D_2^{\delta}\bbFt\|_2 \bigr)\ \leq\ c\,
\bigl(\|\bfft\|_2+\|\bfht\|_{1/2,2;\, \Gammao}\bigr). \label{2.38}
\end{equation}
As the right hand sides of (\ref{2.37}) and (\ref{2.38}) are
independent of $\delta$, we may let $\delta$ tend to $0$ and we
obtain
\begin{equation}
\|\nabla\partial_2\bfvt\|_2+\|\partial_2\pt\|_2\ \leq\ c\,
\bigl(\|\bfft\|_2+\|\bfht\|_{1/2,2;\, \Gammao}\bigr). \label{2.39}
\end{equation}
This shows that $\partial_1\partial_2\vt_1$,
$\partial_2^2\vt_1$, $\partial_1\partial_2\vt_2$,
$\partial_2^2\vt_2$ and $\partial_2\pt$ are all in
$L^2(\Omega)$ and their norms are less than or equal to the
right hand side of (\ref{2.39}). Consequently, as $\bfvt$ is
divergence--free, the same statement also holds on
$\partial_1^2\vt_1$. Now, from equation (\ref{2.12})
(considering just the first scalar component of this vectorial
equation), we deduce that $\partial_1\pt\in L^2(\Omega)$.
Finally, considering the second scalar component in equation
(\ref{2.12}), we obtain $\partial_1^2\vt_2\in L^2(\Omega)$,
too. Thus, applying also (\ref{2.15}) and (\ref{2.16}), we
obtain
\begin{displaymath}
\|\bfvt\|_{2,2}+\|\nabla\pt\|_2\ \leq\ c\, \bigl(
\|\bbF\|_{1,2}+\|\bfg_*\|_{2,2}+\|\bfv\|_{1,2} \bigr).
\end{displaymath}
This inequality, formulas (\ref{2.32}),the estimate of
$\|\bfv_*\|_{2,2}$ and the fact that $\eta=1$ on $\Omega'\equiv
U'\cap\R^2_{d-}$ yield (\ref{2.17}).

\vspace{8pt} \noindent
The next corollary is an immediate consequence of Lemmas
\ref{L2.2} and \ref{L2.6}.

\begin{corollary} \label{C2.1}
Let $\Omega'$ be a sub-domain of $\, \Omega,\, $ such that $\,
\overline{\Omega'}\subset\Omega\cup\Gammai^0\cup\Gammaw\cup
\Gammao^0.\, $ Then $\bfv$ and $p$ satisfy estimate
(\ref{2.17}), where $c=c(\nu,\Omega,\Omega')$.
\end{corollary}

\begin{lemma} \label{L2.7}
Let $\Omega'$ be a sub-domain of $\Omega$, such that
$\overline{\Omega'}\cap\Gammaw=\emptyset$ and $\Gammap\subset
\partial\Omega'$. Then
$\bfv$ and $p$ satisfy estimate (\ref{2.17}), where
$c=c(\nu,\Omega,\Omega')$.
\end{lemma}

\begin{proof}
Consider $\delta\in(0,\tau)$ and denote
\begin{align*}
A_0^{\delta} &:= A_0+\delta\br\bfe_2, & A_1^{\delta} &=
A_1+\delta\br\bfe_2, & B_0^{\delta} &:=B_0+\delta\br\bfe_2,
& B_1^{\delta} &= B_1+\delta\br\bfe_2, \\
\Gammai^{\delta} &:= \Gammai+\delta\br\bfe_2, & \Gammam^{\delta}
&= \Gammam+\delta\br\bfe_2, & \Gammap^{\delta} &:=
\Gammap+\delta\br\bfe_2, & \Gammao^{\delta} &=
\Gammao+\delta\br\bfe_2,
\end{align*}
where $\bfe_2$ is the unit vector in the direction of the
$x_2$--axis. Suppose that $\delta>0$ is so small that
$\Omega'\cap\Gammam^{\delta}=\emptyset$ and the profile $P_0$
lies above $\Gammam^{\delta}$, which means that
$P_0\subset\{(x_1,y_2)\in\R^2;\ y_2>x_2$ for $(x_1,x_2)\in
\Gammam^{\delta}\}$. (Recall that $P_0=\overline{{\rm Int}\,
\Gammaw}$, see Fig.~1.) Denote by $\Omega^{\delta}$ the domain
bounded by the curves $\Gammai^{\delta}$, $\Gammam^{\delta}$,
$\Gammao^{\delta}$, $\Gammap^{\delta}$ and $\Gammaw$.
Precisely,
\begin{displaymath}
\Omega^{\delta}\ :=\ \bigl\{ (x_1,y_2)\in\R^2;\ 0<x_1<d,\
x_2<y_2<x_2+\tau\ \mbox{for}\
(x_1,x_2)\in\Gammam^{\delta}\bigr\} \smallsetminus P_0.
\end{displaymath}
Denote by $\bfv^{\delta}$ the function, defined by the formulas
\begin{equation}
\bfv^{\delta}(x_1,x_2)\ :=\ \left\{ \begin{array}{ll}
\bfv(x_1,x_2) & \mbox{for}\ (x_1,x_2)\in \Omega^{\delta}\cap\Omega, \\
[2pt]  \bfv(x_1,x_2-\tau) & \mbox{for}\ (x_1,x_2)\in
\Omega^{\delta} \smallsetminus\Omega. \end{array} \right.
\label{2.9}
\end{equation}
By analogy, denote
\begin{align*}
\bbF^{\delta}(x_1,x_2)\ &:=\ \left\{ \begin{array}{ll}
\bbF(x_1,x_2) & \mbox{for}\ (x_1,x_2)\in \Omega^{\delta}\cap\Omega, \\
[2pt] \bbF(x_1,x_2-\tau) & \mbox{for}\ (x_1,x_2)\in\Omega^{\delta}
\smallsetminus\Omega, \end{array} \right. \\
\bfg_*^{\delta}(x_1,x_2)\ &:=\ \left\{ \begin{array}{ll}
\bfg_*(x_1,x_2) & \mbox{for}\ (x_1,x_2)\in \Omega^{\delta}\cap\Omega, \\
[2pt] \bfg_*(x_1,x_2-\tau) & \mbox{for}\
(x_1,x_2)\in\Omega^{\delta} \smallsetminus\Omega. \end{array}
\right.
\end{align*}
Let the spaces $\bfV_{\sigma}^{1,2}(\Omega^{\delta})$ and
$\bfV_{\sigma}^{-1,2}(\Omega^{\delta})$ be defined in the same
way as $\Vs$ and $\Vsd$, respectively, and let operator
$\cA^{\delta}$ be defined in the same way as $\cA$, with the
only difference that it acts on functions from
$\bfV_{\sigma}^{1,2}(\Omega^{\delta})$ to
$\bfV_{\sigma}^{-1,2}(\Omega^{\delta})$. Obviously,
$\bbF^{\delta}\in W^{1,2}(\Omega^{\delta})^{2\times 2}$ and
$\|\bbF^{\delta}\|_{1,2;\, \Omega^{\delta}}=\|\bbF\|_{1,2}$.
Similarly, the function $\bfg_*^{\delta}$ has the same norm and
properties in $\Omega^{\delta}$ as the function $\bfg_*$ in
$\Omega$. Let the functionals $\bfF^{\delta}$ and
$\bfG^{\delta}$ in the dual space
$\bfV_{\sigma}^{-1,2}(\Omega^{\delta})$ be defined by analogous
formulas as $\bfF$ and $\bfG$.

Our next claim is to show that
$\bfv^{\delta}\in\bfV_{\sigma}^2(\Omega^{\delta})$ and
$\nu\cA^{\delta}\bfv^{\delta}=\bfF^{\delta}+\nu\,
\bfG^{\delta}$. Since $\bfv\in\Vs$, there exists a sequence
$\{\bfv_n\}$ in $\cCs$, such that $\bfv_n\to\bfv$ in the norm
of $\bfW^{1,2}(\Omega)$. Define
\begin{displaymath}
\bfv_n^{\delta}(x_1,x_2)\ :=\ \left\{ \begin{array}{ll}
\bfv_n(x_1,x_2) & \mbox{for}\ (x_1,x_2)\in
\overline{\Omega^{\delta}} \cap\Omega, \\ [2pt]
\bfv_n(x_1,x_2-\tau) & \mbox{for}\ (x_1,x_2)\in
\overline{\Omega^{\delta}} \smallsetminus\Omega. \end{array}
\right.
\end{displaymath}
Then $\bfv_n^{\delta}\in
\boldsymbol{\cC}_{\sigma}^{\infty}(\overline{\Omega^{\delta}})$
and $\bfv_n^{\delta}\to\bfv^{\delta}$ in
$\bfW^{1,2}(\Omega^{\delta})$. This confirms that
$\bfv^{\delta}\in\bfV_{\sigma}^2(\Omega^{\delta})$.
Furthermore, let $\bfw\in\Vs$ and
$\bfw^{\delta}\in\bfV_{\sigma}^{1,2}(\Omega^{\delta})$ be
related in the same way as $\bfv$ and $\bfv^{\delta}$ in
(\ref{2.9}). Then, denoting by $\langle\, .\, ,\, .\,
\rangle_{\sigma;\, \Omega^{\delta}}$ the duality pairing
between $\bfV_{\sigma}^{-1,2}(\Omega^{\delta})$ and
$\bfV_{\sigma}^{1,2}(\Omega^{\delta})$, we have
\begin{align*}
\langle\nu\cA^{\delta}\bfv^{\delta} &
,\bfw^{\delta}\rangle_{\sigma;\, \Omega^{\delta}} = \nu\,
(\nabla\bfv^{\delta},\nabla \bfw^{\delta})_{2;\, \Omega^{\delta}} \\
&=\ \nu \int_{\Omega^{\delta}\cap\br\Omega}\nabla\bfv^{\delta}:
\nabla \bfw^{\delta}\; \rmd\bfx+\nu\int_{\Omega^{\delta}
\smallsetminus\Omega}\nabla\bfv^{\delta}:\nabla\bfw^{\delta}\;
\rmd\bfx \\
&=\ \nu \int_{\Omega^{\delta}\cap\br\Omega}\nabla\bfv:\nabla\bfw\;
\rmd\bfx+\nu\int_0^{\delta} \int_{\Gammap+\vartheta\,
\rme_2}\nabla\bfv^{\delta}:\nabla\bfw^{\delta}\;
\rmd l\, \rmd\vartheta \\
&=\ \nu\int_{\Omega^{\delta}\cap\br\Omega}\nabla\bfv:\nabla\bfw\;
\rmd\bfx+\nu\int_0^{\delta} \int_{\Gammam+\vartheta\,
\rme_2}\nabla\bfv: \nabla\bfw\; \rmd l\, \rmd\vartheta \\
\noalign{\vskip 4pt}
&=\ \nu\, (\nabla\bfv,\nabla\bfw)_2\ =\ \langle\nu\cA\bfv,\bfw
\rangle_{\sigma}\ =\ \langle\bfF,\bfw\rangle_{\sigma}+\langle\bfG,
\bfw\rangle_{\sigma} \\ \noalign{\vskip 2pt}
&=\ -\int_{\Omega}\bbF:\nabla\bfw\; \rmd\bfx+
\int_{\Omega}\nabla\bfg_*:\nabla\bfw\; \rmd\bfx \\
&=\ -\int_{\Omega^{\delta}} \bbF^{\delta}:\nabla\bfw^{\delta}\;
\rmd\bfx+\int_{\Omega^{\delta}}\nabla\bfg_*^{\delta}:
\nabla\bfw^{\delta}\; \rmd\bfx \\ \noalign{\vskip 4pt}
&=\ \langle\bfF^{\delta},\bfw^{\delta}\rangle_{\sigma;\,
\Omega^{\delta}}+\langle\bfG^{\delta},\bfw^{\delta}
\rangle_{\sigma;\, \Omega^{\delta}}.
\end{align*}
This verifies that $\nu\cA^{\delta}\bfv^{\delta}=
\bfF^{\delta}+\nu\, \bfG^{\delta}$.

Denote $(\Omega')^{\delta/2}:=\Omega'\cup\{(x_1,y_2)\in\R^2;\
x_2\leq y_2<x_2+\frac{1}{2}\delta$ for $(x_1,x_2)\in\Gammap\}$.
Then $(\Omega')^{\delta/2}$ is a sub-domain of
$\Omega^{\delta}$, such that
$\Gammap^0\subset(\Omega')^{\delta/2}$. The statements of Lemma
\ref{L2.7} now follow from Corollary \ref{C2.1}, applied to the
equation $\nu\cA^{\delta}\bfv^{\delta}=
\bfF^{\delta}+\nu\br\bfG^{\delta}$ in domain $\Omega^{\delta}$,
where we consider $(\Omega')^{\delta/2}$ instead of $\Omega'$.
\end{proof}

\vspace{4pt}
An analogue of Lemma \ref{L2.7} also holds if one considers
$\Omega'$, satisfying the condition
$\Gammam\subset\partial\Omega'$ instead of
$\Gammap\subset\partial\Omega'$.  This, Corollary \ref{C2.1},
Lemma \ref{L2.7} and Lemma \ref{L2.4} (which enables us to
estimate $\|\bfv\|_{1,2}$ on the right hand side of
(\ref{2.17})) now imply that (\ref{2.7}) holds.

\vspace{-4pt}
\paragraph{3) \ Solution of the equation $\nu\cA\bfv=\bff$ for
$\bff\in\Ls$.} \ Denote by $D(A)$ the set of functions
$\bfv\in\Vs\cap\bfW^{2,2}_{per}(\Omega)$, such that there
exists $q\in W^{1/2,2}_{per}(\Gammao)$, satisfying
$\partial\bfv/\partial\bfn=q\br\bfe_1$ on $\Gammao$ in the
sense of an equality in $\bfW^{1/2,2}(\Gammao)$. The linear space
$\widetilde{\boldsymbol{\cC}}^{\infty}_{\sigma}(\overline{\Omega}):=
\bigl\{\bfw\in \cCs;\ \partial\bfw/\partial\bfn\perp\bfe_2\
\mbox{on}\ \Gammao\bigr\}$ contains $D(A)$ as a dense subset.
Since $\widetilde{\boldsymbol{\cC}}^{\infty}_{\sigma}
(\overline{\Omega})$ is dense in $\cCs$ in the $L^2$--norm and
$\Ls$ is the closure of $\cCs$ in $\bfL^2(\Omega)$, $D(A)$ is
dense in $\Ls$. Put $A:=\cA\br|_{D(A)}$.

\medskip
Let us at first show that $R(A)$ (the range of $A$) is a subset
of $\Ls$. Thus, let $\bfv\in D(A)$ and $q$ be a corresponding
function in $W^{1/2,2}_{\rm per}(\Gammao)$. It follows from
Lemma \ref{L2.9} that there exists an extension $q_*\in
W^{1,2}_{\rm per}(\Omega)$ of $q$ from $\Gammao$ to $\Omega$,
which equals zero in the neighborhood of $\Gammai$ and
$\Gammaw$ and satisfies
\begin{equation}
\|q_*\|_{1,2}\ \leq\ c\, \|q\|_{1/2,2;\, \Gammao}, \label{2.46}
\end{equation}
where $c=c(\Omega)$. For any $\bfw\in\Vs$, $\bfv$ satisfies
\begin{align*}
\langle A\bfv,\bfw\rangle_{\sigma}\ &=\ (\nabla\bfv,
\nabla\bfw)_2\ = \int_{\Gammao}\frac{\partial\bfv}
{\partial\bfn}\cdot\bfw\; \rmd l-(\Delta\bfv,\bfw)_2 \\
&=\ \int_{\Gammao}q\br\bfn\cdot\bfw\; \rmd l- (\Delta\bfv,\bfw)_2\
=\ (\nabla q_*-\Delta\bfv,\bfw)_2.
\end{align*}
From this and the density of $\Vs$ in $\Ls$, we deduce that
$A\bfv$ can be identified with a bounded linear functional on $\Ls$.
Due to Riesz' theorem, it can be represented by an element of $\Ls$
(which is, in our case, the function $\nabla q_*-\Delta\bfv$).
We have proven the inclusion $R(A)\subset\Ls$.

Treating $A$ as an operator in $\Ls$, we easily verify that $A$
is symmetric: let $\bfv^{(1)}$, $\bfv^{(2)}\in D(A)$. Then
\begin{align*}
\bigl(A\bfv^{(1)},\bfv^{(2)}\bigr)_2\ &=\
\blangle\cA\bfv^{(1)},\bfv^{(2)}\brangle_{\sigma}\ =\
\bigl(\nabla\bfv^{(1)}, \nabla\bfv^{(2)}\bigr)_2\ =\
\bigl(\nabla\bfv^{(2)}, \nabla\bfv^{(1)}\bigr)_2 \\
&=\ \blangle\cA\bfv^{(2)},\bfv^{(1)}\brangle_{\sigma}\ =\
\bigl(A\bfv^{(2)},\bfv^{(1)}\bigr)_2\ =\
\bigl(\bfv^{(1)},A\bfv^{(2)}\bigr)_2.
\end{align*}

Further, we show that operator $A$ is closed: let $\{\bfv_n\}$
be a sequence in $D(A)$, such that $\bfv_n\to\bfv$ (for
$n\to\infty$) in $\Ls$. Put $\bff_n:=A\bfv_n$. Suppose that
$\bff_n\to\bff$ in $\Ls$. Put $\bbF_n:=\cF(\bff_n,\bfzero)$,
where $\cF$ is the operator from Lemma \ref{L2.8}. As all
functions $\bfv_n$ ($n=1,2,\dots$) lie in
$\Vs\cap\bfW^{2,2}(\Omega)$, we may apply estimate (\ref{2.7})
(where we consider $\bfgs=\bfzero$) to the difference
$\bfv_m-\bfv_n$ (for any $m,n\in\N$) and afterwards use the
boundedness of operator $\cF$:
\begin{displaymath}
\|\bfv_m-\bfv_n\|_{2,2}\ \leq\ c\, \|\bbF_m-\bbF_n\|_{1,2}\ =\
c\, \| \cF(\bff_m,\bfzero)-\cF(\bff_n,\bfzero)\|_{1,2}\ \leq\
c\, \|\bff_m-\bff_n\|_2,
\end{displaymath}
where $c$ is independent of $m,\, n$. From this, we deduce that
$\bfv\in\Vs\cap\bfW^{2,2}(\Omega)$. As
$\partial\bfv_n/\partial\bfn$ is in $W^{1/2,2}_{per}(\Gammao)$
and normal to $\Gammao$ (for each $n\in\N$), there exists $q\in
W^{1/2,2}_{per}(\Gammao)$, such that
$\partial\bfv/\partial\bfn=q\br\bfn$ on $\Gammao$. Hence
$\bfv\in D(A)$ and $A\bfv=\bff$. We have proven that $A$ is a
closed operator in $\Ls$.

Operator $A$ is positive, because
$(A\bfv,\bfv)=\|\nabla\bfv\|_2^2$ for $\bfv\in D(A)$.
Consequently, $A$ is a self-adjoint operator in $\Ls$. Then
$R(A)^{\perp}$ (the orthogonal complement to $R(A)$ in $\Ls$)
is equal to $N(A)$ (the null space of $A$), see
\cite[p.~168]{Ka}. However, as $A\subset\cA$ and
$N(\cA)=\{\bfzero\}$, we also have
$R(A)^{\perp}=N(A)=\{\bfzero\}$. This shows that $R(A)$ is
dense in $\Ls$. As all functions from $D(A)$ are in
$\Vs\cap\bfW^{2,2}(\Omega)$, we may again apply estimate
(\ref{2.7}) (with $\bfgs=\bfzero$) and afterwards the open
graph theorem  and deduce that $R(A)$ is closed in $\Ls$. Thus,
$R(A)=\Ls$. As $A$ coincides with $\cA$ on $D(A)$, we observe
that if $\bff\in\Ls$ (which can be identified with a subspace
of $\Vsd$), the equation $\nu\cA\bfv=\bff$ has a solution in
$D(A)$.

\vspace{-4pt}
\paragraph{4) \ Solution of the equation $\nu\cA\bfv_1=\bfF_1$.} \
Recall that the functional $\bfF_1\in\Vsd$ is defined by
formula (\ref{2.1}), where $\bbF_1\cdot\bfn=h_1\br\bfn$ on
$\Gammao$ and $h_1\in W^{1/2,2}_{\rm per}(\Gammao)$.
Extending function $h_1$ from $\Gammao$ to $\Omega$ (by means
of Lemma \ref{L2.9}) so that the extended function $h_{1*}$ is
in $W^{1,2}_{\rm per}(\Omega)$ and equals zero in the
neighborhood of $\Gammai$ and $\Gammaw$, we obtain
\begin{align*}
\langle\bfF_1,\bfw\rangle_{\sigma}\ &=\
-\int_{\Omega}\bbF_1:\nabla\bfw\; \rmd\bfx\ =\
-\int_{\Gammao}(\bbF_1\cdot\bfn)\cdot\bfw\; \rmd
l+\int_{\Omega}\div\bbF_1\cdot\bfw\; \rmd\bfx \\
&=\ -\int_{\Gammao}h_1\br\bfn\cdot\bfw\; \rmd
l+\int_{\Omega}\div\bbF_1\cdot\bfw\; \rmd\bfx\ =\ \int_{\Omega}
[-\nabla h_{1*}+\div\bbF_1]\cdot\bfw\; \rmd\bfx \\ \noalign{\vskip 4pt}
&=\ \bigl(-\nabla h_{1*}+\div\bbF_1,\br\bfw\bigr)_2.
\end{align*}
From this, we observe that $\bfF_1$ can be identified with a
function from $\Ls$. The inclusion $\bfv_1\in\Vs\cap
\bfW^{2,2}_{\rm per}(\Omega)$  now follows from part 3) of this
proof.

Due to Theorem \ref{T1}, there exists $p_1\in L^2(\Omega)$, such that
$\bfv_1$ and $p_1$ satisfy equation (\ref{2.33}), which now
takes the form
\begin{equation}
-\nu\Delta\bfv_1+\nabla p_1+\div\bbF_1\ =\ \bfzero. \label{2.24}
\end{equation}
As $\Delta\bfu$ and $\div\bbF_1$ belong to $\bfL^2(\Omega)$,
$\nabla p_1$ is in $\bfL^2(\Omega)$, too. Thus, $p_1\in
W^{1,2}(\Omega)$. Let us show that $p_1\in W^{1,2}_{\rm
per}(\Omega)$. Multiplying equation (\ref{2.24}) by
$\bfw\in\Vs$, we obtain
\begin{align*}
0\ &=\ \int_{\Omega}\bigl[-\nu\Delta\bfv_1+\nabla
p_1+\div\bbF_1\bigr]\cdot\bfw\; \rmd\bfx\ =\ \int_{\partial\Omega}
\Bigl[-\nu\, \frac{\partial\bfv_1}{\partial\bfn}+p_1\br\bfn+
\bbF_1\cdot\bfn\Bigr]\cdot\bfw\; \rmd l \\
&=\ \int_{\Gammam\cup\Gammap} \Bigl[-\nu\,
\frac{\partial\bfv_1}{\partial\bfn}+p_1\br\bfn+
\bbF_1\cdot\bfn\Bigr]\cdot\bfw\; \rmd l+\int_{\Gammao}
\Bigl[-\nu\, \frac{\partial\bfv_1}{\partial\bfn}+p_1\br\bfn+
\bbF_1\cdot\bfn\Bigr]\cdot\bfw\; \rmd l \\
&=\ \int_{\Gammam\cup\Gammap}p_1\, \bfn\cdot\bfw\; \rmd l\ =\
\int_{\Gammam}\bigl[ p_1(x_1,x_2)-p(x_1,x_2+\tau)\bigr]\,
\bfn\cdot\bfw\; \rmd l.
\end{align*}
Since this holds for all $\bfw\in\Vs$, $p_1$
satisfies the condition of periodicity (\ref{1.7}).

\vspace{-4pt}
\paragraph{5) \ Solution of the equation $\nu\cA\bfv_2=\bfF_2$.} \
One can deduce by means of Lemma \ref{L2.9} that, there exists
a function $\psi\in W^{3,2}_{\rm per}(\Omega)$, such that
$\psi=\partial_1\psi=0$ on $\Gammao$ and $\nu\,
\partial_1^2\psi=h_2$ on $\Gammao$. Function $\psi$ equals zero in
the neighborhood of $\Gammai$ and $\Gammaw$. Put
$\bfv_2:=-\nabla^{\perp}\psi$. Then
$\bfv_2\in\Vs\cap\bfW^{2,2}_{\rm per}(\Omega)$ and $\nu\,
\partial\bfv_2/\partial\bfn=-h_2\, \bfe_2$ on $\Gammao$. Put $\bff_2:=
-\nu\Delta\bfv_2$ and $\bbH:=\cF(-\div\bbF_2+\bff_2,\bfzero)$,
where $\cF$ is the operator from Lemma \ref{L2.8}. Then
$\div\bbH=-\div\bbF_2+\bff_2$ in $\Omega$ and
$\bbH\cdot\bfn=\bfzero$ on $\Gammao$. Now, for all
$\bfw\in\Vs$, we have
\begin{align*}
\nu\, \blangle\nabla\bfv_2 & , \nabla\bfw\brangle_2\ =
\int_{\Gammao}\!\! \nu\,
\frac{\partial\bfv_2}{\partial\bfn}\cdot\bfw\; \rmd
l-\int_{\Omega}\nu\Delta\bfv_2\cdot\bfw\; \rmd\bfx \\
&=\ -\int_{\Gammao}h_2\br\bfe_2\cdot\bfw\; \rmd
l+\int_{\Omega}\bff_2\cdot\bfw\; \rmd\bfx \\
&=\ -\int_{\Gammao}(\bbF_2\cdot\bfn)\cdot\bfw\; \rmd
l+\int_{\Omega}\div(\bbF_2+\bbH)\cdot \bfw\; \rmd\bfx \\
&=\ -\int_{\Gammao}(\bbF_2\cdot\bfn)\cdot\bfw\; \rmd
l+\int_{\Omega}\div\bbF_2\cdot \bfw\; \rmd\bfx\ =\ -\int_{\Omega}
\bbF_2:\nabla\bfw\; \rmd\bfx \\
&=\ \langle\bfF_2,\bfw\rangle_{\sigma}.
\end{align*}
This shows that $\nu\cA\bfv_2=\bfF_2$. In other words, we have
proven that this equation has a solution in
$\Vs\cap\bfW^{2,2}_{\rm per}(\Omega)$. By analogy with $p_1$,
there exists an associated pressure $p_2\in W^{1,2}_{\rm per}
(\Omega)$.

\vspace{-4pt}
\paragraph{6)\ Solution of the equation $\nu\cA\bfv_3=\nu\bfG$.} \ As
$\Delta\bfgs\in\bfL^2(\Omega)$, we may put
$\bbF_*:=\cF(\nu\Delta\bfgs,\bfzero)$. Then
$\div\bbF_*=\nu\Delta\bfgs$ in $\Omega$ and
$\bbF_*\cdot\bfn=\bfzero$ on $\Gammao$. Let the functional
$\bfF_*\in\Vsd$ be defined by formula (\ref{2.1}), where we
consider $\bbF_*$ instead of $\bbF$. The solution $\bfv_3$ of
the equation $\nu\cA\bfv_3=\bfG$ satisfies
\begin{align*}
\langle\nu\cA\bfv_3,\bfw\rangle_{\sigma}\ &=\
\langle\nu\bfG,\bfw\rangle_{\sigma}\ =\
-\int_{\Omega}\nu\nabla\bfgs:\nabla\bfw\; \rmd\bfx\ =\
\int_{\Omega} \nu\Delta\bfgs\cdot\bfw\; \rmd\bfx \\
&=\ \int_{\Omega}\div\bbF_*\cdot\bfw\; \rmd\bfx\ =\ -\int_{\Omega}\bbF_*:\nabla\bfw\; \rmd\bfx
\end{align*}
for all $\bfw\in\Vs$. Since $\bbF_*$ is in $W^{1,2}_{\rm
per}(\Omega)^{2\times 2}$, we obtain the inclusion $\bfv_3\in
\Vs\cap\bfW^{2,2}_{\rm per}(\Omega)$ from part 4) or 5) of this
proof. The existence of an associated pressure $p_3\in
W^{1,2}_{\rm per}(\Omega)$ now follows by means of the same
arguments as at the end of part 4).

\vspace{-4pt}
\paragraph{7) \ The validity of statement (a).} \ The solvability of
the equation $\nu\cA\bfv=\bfF+\nu\bfG$ in
$\Vs\cap\bfW^{2,2}_{\rm per}(\Omega)$ and the existence of an
associated pressure in $W^{1,2}_{\rm per}(\Omega)$ now follows
from the decomposition of the right hand side to
$\bfF_1+\bfF_2+\nu\bfG$ and from the parts 4), 5) and 6) of
this proof.

The proof of Theorem \ref{T2} is completed.

\begin{remark} \label{R2} \rm
Theorem \ref{T2} can be generalized so that instead of the
functions $\bfu\in\bfW^{2,2}_{\rm per}(\Omega)$ and $p\in
W^{1,2}_{\rm per} (\Omega)$, it yields
$\bfu\in\bfW^{n+2,2}_{\rm per}(\Omega)$ and $p\in
W^{n+1,2}_{\rm per}(\Omega)$ for $n\in\{0\}\cup\N$. The
generalization says:

\vspace{2pt}
{\it Let $n\in\N\cup\{0\}$. Let the closed curve $\Gammaw$ (the
boundary of profile $P_0$) be of the class $C^{n+2}$, $\bbF\in
W^{n+1,2}_{\rm per}(\Omega)^{2\times 2}$ and $\bfg$, $\bfg_*$
be the functions from Lemma \ref{L2.1}, where we consider
$m=n+1$. Let the functionals $\bfF$ and $\bfG$ be defined by
formulas (\ref{2.1}) and (\ref{2.34}), respectively. Then

\begin{list}{}
{\setlength{\topsep 1pt}
\setlength{\itemsep 1pt}
\setlength{\leftmargin 19pt}
\setlength{\rightmargin 0pt}
\setlength{\labelwidth 14pt}}

\item[1)]
the unique solution $\bfv$ of the equation
$\nu\cA\bfv=\bfF+\nu\br\bfG$ belongs to the space
$\Vs\cap\bfW^{n+2,2}_{\rm per}(\Omega)$ and the associated
pressure $p$ is in $W^{n+1,2}_{\rm per} (\Omega)$,

\item[2)]
$\bfu$, $p$ satisfy statements (b) and (c) of Theorem \ref{T2},

\item[3)]
there exists a constant $\cn02=\cc02(\nu,\Omega,n)$, such that
\begin{equation}
\|\bfu\|_{n+2,2}+\|\nabla p\|_{n,2}\ \leq\ \cc02\, \bigl(
\|\bbF\|_{n+1,2}+\|\bfg_*\|_{n+2,2} \bigr). \label{2.26}
\end{equation}

\end{list}}

As the complete proof of the generalization would be long and
its steps would be just technical modifications of the steps
from the proof of Theorem \ref{T2}, we do not include it here. We only
note that the corresponding analogue of Lemma \ref{L2.2} would
use Proposition I.2.3 from \cite{Te} with $n+2$ instead of $2$,
the analogue of Lemma \ref{L2.6} would use Theorem III.3.3 from
\cite{Ga} in a subtler way and with $m=n+1$ instead of $m=1$ in
order to obtain function $\bfv_*$ (see the proof of Lemma
\ref{L2.6}), and it would be also necessary to use higher order
difference quotients in the proof of the analogue of Lemma
\ref{L2.6}.
\end{remark}

\vspace{-8pt}
\paragraph{Acknowledgement.} The author acknowledges the support of
the European Regional Deve\-lopment Fund-Project ``Center for
Advanced Applied Science'' No.~CZ.02.1.01/0.0/0.0/ \\ 16\_019
/0000778.

\end{document}